\RequirePackage{auxhook}
\PassOptionsToPackage{hyphens}{url}
\documentclass[twoside,leqno
]{article}
\usepackage[a5paper,layoutwidth=148.5mm,layoutheight=7.77817in,margin=36pt,
footskip=19pt
]{geometry}

\usepackage[utf8]{inputenc} 
\usepackage[T1]{fontenc}
\usepackage{lmodern}
\usepackage{microtype}
\usepackage{textcomp}

\usepackage{natbib}
\bibpunct{(}{)}{;}{a}{}{,}
\newcommand\urlprefix{}
\usepackage[german,french,british]{babel}

\usepackage{amssymb}
\usepackage{mathtools}
\mathtoolsset{mathic}

\RequirePackage{etoolbox}
\makeatletter
\patchcmd{\@footnotetext}{\unskip\strut}{\unskip\@finalstrut\strutbox}{}{}
\makeatother

\usepackage{xspace}
\xspaceaddexceptions{]}

\usepackage{paralist}

\newcommand\Bullet{\phantom{(1)}\llap{\textbullet}}

\DeclareRobustCommand{\eoe}{\leavevmode\unskip\penalty9999\hbox{}\nobreak\hfill\quad$\diamond$}

\makeatletter

\DeclareFontEncoding{LS1}{}{}
\DeclareFontEncoding{LS2}{}{\noaccents@}
\DeclareFontSubstitution{LS1}{stix}{m}{n}
\DeclareFontSubstitution{LS2}{stix}{m}{n}
\DeclareSymbolFont{stix-operators}     {LS1}{stix}     {m} {n}
\DeclareSymbolFont{stix-largesymbols}  {LS2}{stixex}   {m} {n}
\SetSymbolFont{stix-operators}   {bold}{LS1}{stix}     {b} {n}
\SetSymbolFont{stix-largesymbols}{bold}{LS2}{stixex}   {b} {n}

\DeclareMathSymbol{\stixwedge}     {\mathbin}  {stix-operators}   {"E1}
\DeclareMathSymbol{\stixvee}       {\mathbin}  {stix-operators}   {"E2}
\DeclareMathSymbol{\stixbigwedgeop}{\mathop}   {stix-largesymbols}{"B4}
\DeclareMathSymbol{\stixbigveeop}  {\mathop}   {stix-largesymbols}{"B5}
\DeclareMathDelimiter{\stixlBrack} {\mathopen} {stix-largesymbols}{"E0}{stix-largesymbols}{"06}
\DeclareMathDelimiter{\stixrBrack} {\mathclose}{stix-largesymbols}{"E1}{stix-largesymbols}{"07}

\newcommand\Vii{\mathop{\boldsymbol{\stixbigwedgeop}}\slimits@}
\newcommand\Vuu{\mathop{\boldsymbol{\stixbigveeop}}\slimits@}
\newcommand\lrb[1]{\stixlBrack #1 \stixrBrack}
\newcommand\vii{\boldsymbol{\stixwedge}}
\newcommand\vuu{\boldsymbol{\stixvee}}

\def\myeqctr#1{\expandafter\@myeqctr\csname c@#1\endcsname}
\def\@myeqctr#1{%
  \ifcase#1\or
  \TextOrMath\textasteriskcentered *\or
  \TextOrMath \textdagger \dagger\or
  \ensuremath{\mkern1mu\flat}\or
  \ensuremath{\natural}\or
  \ensuremath{\sharp}\or
  \TextOrMath \textdaggerdbl \ddagger \or
  \TextOrMath \textsection  \mathsection\or
  \TextOrMath \textbardbl \|\or
  \ensuremath{\mkern1mu^\circ}\or
  \textup\textreferencemark\or
  \TextOrMath \textparagraph \mathparagraph\or
  \#\or
  \%\or
  \textborn\or
  \#\#\or
  \%\%\or
  \TextOrMath {\textasteriskcentered\textasteriskcentered}{**}\or
  \TextOrMath {\textdagger\textdagger}{\dagger\dagger}\or
  \TextOrMath {\textdaggerdbl\textdaggerdbl}{\ddagger\ddagger}\or
  \TextOrMath {\textsection\textsection}  {\mathsection\mathsection}\or
  \ensuremath{\mkern1mu^{\circ\circ}}\or
  \ensuremath{\mkern1mu\flat\flat}\or
  \ensuremath{\natural\natural}\or
  \ensuremath{\sharp\sharp}\or
  \textup{\textreferencemark\textreferencemark}\or
  \TextOrMath {\textparagraph\textparagraph} {\mathparagraph\mathparagraph}\else
  \@ctrerr \fi
}

\makeatother

\DeclareMathOperator\Icl{Icl}

\newcommand\Equivariance{Equivariance\xspace}
\newcommand\Equivariant{Equivariant\xspace}
\newcommand\Indtrs{Meet-semi\-lattices\xspace}
\newcommand\Indtr{Meet-semi\-lattice\xspace}

\newcommand\Pfs{\mathrm{P}_{\phantom{*}\llap{$\scriptstyle\mathrm{fe}$}}^*}

\newcommand\Property{Property\xspace}
\newcommand\QQ{\mathbb{Q}} 

\newcommand\Rule[1]{\tag*{\textit{#1}}\label[property]{#1}}

\newcommand\Sis{\Equivariant systems of ideals\xspace}
\newcommand\Trdis{Distributive lattices\xspace}

\newcommand\entrels{entailment relations\xspace}
\newcommand\entrel{entailment relation\xspace}
\newcommand\eqdef{\mathrel{\overset{\makebox[0pt]{\mbox{\normalfont\tiny def}}}{=}}}
\newcommand\equidef{\mathrel{\overset{\makebox[0pt]{\mbox{\normalfont\tiny def}}}{\Longleftrightarrow}}}
\newcommand\equivariance{equivariance\xspace}
\newcommand\equivariant{equi\-va\-riant\xspace}
\newcommand\eti{^\times}
\newcommand\etl{^*}
\newcommand\fa{\mathfrak{a}}
\newcommand\fb{\mathfrak{b}}
\newcommand\fc{\mathfrak{c}}
\newcommand\gen[1]{\mathchoice{{\left\langle{#1}\right\rangle}}{\langle #1\rangle}{\langle #1\rangle}{\langle #1\rangle}} 
\newcommand\grls{\texorpdfstring{\(\ell\)}l-groups\xspace}
\newcommand\grl{\texorpdfstring{\(\ell\)}l-group\xspace}
\newcommand\gui[1]{``#1''}
\newcommand\id{R} % integral domain
\newcommand\indtrs{meet-semi\-lattices\xspace}
\newcommand\indtr{meet-semi\-lattice\xspace}
\newcommand\lat{L}
\newcommand\lrbm{\lrb{1..m}}
\newcommand\lt{\preccurlyeq}
\newcommand\mmonoids{meet-monoids\xspace}
\newcommand\mmonoid{meet-monoid\xspace}
\newcommand\mor{\varphi}
\newcommand\ndsp{\textstyle}
\newcommand\presord{preservation of order\xspace}
\newcommand\properties{properties\xspace}
\newcommand\property{property\xspace}
\newcommand\rP{\mathrm{P}}
\newcommand\rh[1][]{\mathrel{{\rhd}_{#1}}}

\newcommand\rsi{regular system of ideals}
\newcommand\scentrels{systems of ideals\xspace}
\newcommand\scentrel{system of ideals\xspace}
\newcommand\set{G}
\newcommand\sis{\equivariant systems of ideals\xspace}
\newcommand\si{\equivariant system of ideals\xspace}
\newcommand\slat{S}
\newcommand\so[1]{\left\{{#1}\right\}} 
\newcommand\som{\sum\nolimits}
\newcommand\sotq[2]{\so{\,#1\mid#2\,}} 
\newcommand\trdis{distributive lattices\xspace}
\newcommand\trdi{distributive lattice\xspace}
\newcommand\vda[1][]{\vdash_{#1}}
\newcommand\Vda{\rh[\!\mathrm{a}]}

\renewcommand\geq{\geqslant}
\renewcommand\leq{\leqslant}
\renewcommand\Delta{\mathsf{\mathchar"7001}}
\renewcommand\Gamma{\mathsf{\mathchar"7000}}

\usepackage[pdfpagelabels=true]{hyperref}
\hypersetup{%hidelinks,
  bookmarksopen=false,breaklinks=true,
    plainpages=false,
  colorlinks=true,linkcolor=blue,
  citecolor=red,urlcolor=red,hypertexnames=false}
\providecommand\phantomsection{}
\providecommand\texorpdfstring[2]{#1}

\usepackage{doi}

\RequirePackage{url}

\usepackage{breakurl}

\RequirePackage{amsthm}

\usepackage[capitalise]{cleveref}

\crefname{property}{Property}{Properties}
\crefname{inequality}{Inequality}{Inequalities}
\crefname{equation}{}{}
\creflabelformat{enumi}{(#2#1#3)}
\creflabelformat{inequality}{(#2#1#3)}
\renewcommand\cpageref[1]{page~\pageref{#1}}

\theoremstyle{plain}
\newtheorem{theorem}{Theorem}[section]
\newtheorem*{theorem*}{Theorem}
\newtheorem{satzeins}{Satz}
\newtheorem{lemma}[theorem]{Lemma}
\newtheorem{corollary}[theorem]{Corollary}
\newtheorem{proposition}[theorem]{Proposition}

\theoremstyle{definition}
\newtheorem{definition}[theorem]{Definition}
\newtheorem*{acknowledgement}{Acknowledgement}

\theoremstyle{remark}
\newtheorem{remark}[theorem]{Remark}
\newtheorem{remarks}[theorem]{Remarks}
\newtheorem{comment}[theorem]{Comment}
\newtheorem{comments}[theorem]{Comments}
\newtheorem{example}[theorem]{Example}

\date{}
\title{Lattice-ordered groups generated by an ordered group and regular systems of ideals}

\author{Thierry Coquand%}
\and
Henri Lombardi%}
\and
Stefan Neuwirth}
\begin{document}
\tolerance 1414
\hbadness 1414
\emergencystretch 1.5em
\hfuzz 0.3pt
\widowpenalty=10000

\maketitle

\begin{abstract}
  Unbounded entailment relations, introduced by Paul \citet{Lor1951},
  are a slight variant of a notion which plays a fundamental rôle in
  logic \citep[see][]{Sco1974} and in algebra \citep[see][]{CACM}. We
  call \emph{\scentrels} their single-con\-clu\-sion counterpart. If
  they preserve the order of a commutative ordered monoid~$G$ and are
  \equivariant w.r.t.\ its law, we call them \emph{\sis for~$G$}: they
  describe all morphisms from~$G$ to meet-semilattice-ordered monoids
  generated by (the image of)~$G$.  Taking an article by
  \citet{Lor1953} as a starting point, we also describe all morphisms
  from a commutative ordered group~$G$ to lattice-ordered groups
  generated by~$G$ through unbounded entailment relations that
  preserve its order, are \equivariant, and satisfy a regularity
  \property invented by \citet{Lor1950}; we call them \emph{regular
    entailment relations}.  In particular, the free lattice-ordered
  group generated by~$G$ is described through the finest regular
  entailment relation for~$G$, and we provide an explicit description
  for it; it is order-reflecting if and only if the morphism is
  injective, so that the Lorenzen-Clifford-Dieudonné theorem fits into
  our framework.  Lorenzen's research in algebra starts as an inquiry
  into the system of Dedekind ideals for the divisibility group of an
  integral domain~$\id$, and specifically into
  \foreignlanguage{german}{Wolfgang Krull}'s
  ``\foreignlanguage{german}{Fundamentalsatz}'' that $\id$ may be
  represented as an intersection of valuation rings if and only if
  $\id$~is integrally closed: his constructive substitute for this
  representation is the \emph{regularisation} of the system of
  Dedekind ideals, i.e.\ the lattice-ordered group generated by it
  when one proceeds as if its elements are comparable.
\end{abstract}

\smallskip


%%%%%%%%%%%%%%%%%%%%%%%%%%%%%%%%%%%%%%%%%%%%%%%%%%%%%%%%%%%%%%%%%%%%
\section*{Introduction}

In this article, all monoids and groups are supposed to be
commutative, and orders are tacitly partial.

The idea of generating a semilattice and a distributive lattice by a
logic-free and set-theory-free formal system, called respectively
``\scentrel'' and ``unbounded entailment relation'' in this article,
dates back to \citet[\S2]{Lor1951} and is motivated there as capturing
how ideal theory provides formal gcds and lcms, i.e.\ formal meets and
joins, for elements of an integral domain. Multiplicative ideal theory
gives rise to ``equivariant'' counterparts to these formal systems.

After studying \citealt{Lor1953}, we have isolated a new axiom that we
call ``regularity''. In this article, our aim is to give a precise
account of Lorenzen's results through ``regular'' entailment
relations. Our main theorem, \cref{ithGOEntrelGRL}, shows that by means of this
axiom, an equivariant entailment relation generates an \grl.

\citet{Lor1950} introduces a construction that embodies the right to
compute in an \si as if it was linearly ordered; we formulate it as
``regularisation'' in \cref{ideflorrel}. \Cref{embedLorgroup} states
that this gives rise to an \grl, the ``Lorenzen group'' associated
with the \si.  The literature on \grls seems not to have taken notice
of these results.

In Lorenzen's work, this approach supersedes another, based on a
procedure for forcing the cancellativity of an \si, ideated by
\citet{Pru1932} and generalised to the setting of ordered monoids in
\citeauthor{Lor1939}'s Ph.D.\ thesis \citeyearpar{Lor1939}. In
\cref{secGoembedsGrl}, we also provide an account for that.

The key step in our presentation is to show that a regular entailment
relation defines by restriction a cancellative \si; in both approaches, the
sought-after \grl is constructed as the Grothendieck \grl of a
cancellative monoid of ideals (\cref{thGoembedsGrl}).

\subsection*{The \foreignlanguage{german}{Fundamentalsatz} for integral domains}
\label{sec:fore-integr-doma-1}

The motivating example for Lorenzen's analysis of the concept of ideal
is \foreignlanguage{german}{Wolfgang Krull}'s
``\foreignlanguage{german}{Fundamentalsatz}'', which states that an
integral domain is an intersection of valuation rings if and only if
it is integrally closed. As \citet[page~111]{Kru35} himself
emphasises, ``Its main defect, that one must not overlook, lies in
that it is a purely existential theorem'', resulting from a
well-ordering argument. In a letter to Heinrich
Scholz\footnote{Scholz-Archiv, Universitäts- und Landesbibliothek Münster,
  \url{http://www.uni-muenster.de/IVV5WS/ScholzWiki/doku.php?id=scans:blogs:ko-05-0647},
  accessed 21st September 2016, published in \citealt[§~V]{neuwirthkonstanz}.} dated 18th April 1953, Krull writes:
``At working with the uncountable, in particular with the
well-ordering theorem, I always had the feeling that one uses fictions
there that need to be replaced some day by more reasonable
concepts. But I was not getting upset over it, because I was convinced
that at a careful application of the common `fictions' nothing false
comes out, and because I was firmly counting on the man who would some
day put all in order. Lorenzen has now found according to my
conviction the right way [\dots]''.

Lorenzen shows that the well-ordering argument in Krull's proof may be
replaced by the performance of computations as if the monoid of
Dedekind ideals was linearly ordered (see \cref{remhist3}), that
integral closedness guarantees that such computations do not add new
relations of divisibility to the integral domain, and that this
performance, formulated as regularisation, generates a lattice-ordered
group. \Cref{embedLorgroup} is in fact an abstract version of the
following theorem (see \cref{thmDivLor2}).
\begin{theorem*}
  The divisibility group of an integral domain embeds into an \grl
  that contains the system of Dedekind ideals if and only if the
  integral domain is integrally closed.
\end{theorem*}

\subsection*{Outline of the article}
\label{sec:outline-article}

\Cref{TrdiNonBornes} deals with \indtrs as generated by \scentrels,
discusses \sis for an ordered monoid and the \mmonoid they generate
(\cref{ThSIJaf}).  \Cref{sec:grls-regular-entrels} deals with \trdis
as generated by unbounded \entrels and discusses regular entailment
relations. \Cref{sec:cons-canc} introduces the Grothendieck \grl of a
meet-monoid as a means for proving \cref{ithGOEntrelGRL}.
\Cref{sec:regularisation} investigates regularisation: applied to the
finest \si, it leads to the finest regular \entrel and to the \grl
freely generated by an ordered group
(\cref{secSysId,subseclcdalaLor,sec:lorenz-cliff-dieud}); applied to
the system of Dedekind ideals for the divisibility group of an
integral domain, it captures the concept of integral dependence and
leads to Lorenzen's theory of divisibility
(\cref{sec:fore-integr-doma,sec:forc-posit-an-1,subsecaritalaLor,secalaLor}).
\Cref{secGoembedsGrl} reminds us of an important theorem by Prüfer
which has led to the historically first approach to the Lorenzen group
associated with an \si.

This article is written in Errett Bishop's style of constructive
mathematics \citep*{B67,BR1987,MRR,CACM}: all theorems can be viewed
as providing an algorithm that constructs the conclusion from the
hypotheses.

%%%%%%%%%%%%%%%%%%%%%%%%%%%%%%%%%%%%%%%%%%%%%%%%%%%%%%%%%%%%%%%%%%%%
%%%%%%%%%%%%%%%%%%%%%%%%%%%%%%%%%%%%%%%%%%%%%%%%%%%%%%%%%%%%%%%%%%%%
\section{\Indtr-ordered monoids and \sis}\label{TrdiNonBornes}
\nopagebreak\subsection{\Indtrs and \scentrels}

Let us define a \emph{\indtr} as a purely equational algebraic
structure with just one law~$\vii$ that is idempotent, commutative,
and associative. We are leaving out the axiom of \indtrs providing a
greatest element because it does not suit monoid theory: meets are
only supposed to exist for \emph{nonempty} finitely enumerated sets.

Let $\Pfs(\set)$~be the set of nonempty finitely enumerated subsets of
an arbitrary set~$\set$.  For a \indtr~$\slat$, let us denote by
$A\rh b$ the relation defined between the sets~$\Pfs(\slat)$
and~$\slat$ in the following way
\citep[see][Satz~1]{Lor1951}:\footnote{The sign~$\rhd$ has been
  introduced with this meaning and with the terminology
  ``single-conclusion entailment relation'' by
  \citet*{rinaldischusterwessel16}.}
\[
  A\rh b \enskip \equidef\enskip \ndsp\Vii A\leq_\slat b\enskip \equidef\enskip b\vii{\Vii A} =_\slat \Vii A\text.
\]
This relation is reflexive, monotone (a property also called
``thinning'' and ``weakening''), and transitive (a \property also
called ``cut'' because it ``cuts''~$c$) in the following sense,
expressed without the law~$\vii$:
\begin{alignat*}{3}
  &&a&\rh a&\text{(reflexivity);}\Rule{S0}\\
  &\text{if $A\rh b$, then }&A,A'&\rh b&\qquad\text{(monotonicity);}\Rule{S1}\\
  \text{if $A\rh c$ and }&A,c\rh b\text{, then }&A&\rh b&\text{(transitivity).}\Rule{S2}
\end{alignat*}
Note that in the context of relations, we shall make the following
abuses of notation for finitely enumerated sets: we write $a$ for the
singleton consisting of~$a$, and $A,A'$ for the union of the sets~$A$
and~$A'$.  These three properties correspond respectively to the
``tautologic assertions'', the ``immediate deductions'', and to an
elementary form of the ``syllogisms'' of the systems of axioms
introduced by Paul \citet[\S~1]{Her1923}, so that the following definition
may be attributed to him;\footnote{\label{beziau}Jean-Yves
  \citet[\S~6]{Bez2006} discusses the relationship of \scentrels with
  Alfred Tarski's consequence operation, which may be compared to the
  relationship of our \cref{defsysideals} of an \si with the
  set-theoretic star-operation: see \cref{remJafsi2} of
  \cref{remJafsi}.} see also Gerhard \citet[\S~2]{Gen1933}, who has coined
the vocables ``thinning'' and ``cut''.  This definition is
introduced as description of a meet-semilattice (see
\cref{thSemiEntRelNB}) in \citet[\S~2]{Lor1951}.
\begin{definition}
  \label{defEntrel-ter}
  A \emph{\scentrel} for a set~$G$ is a reflexive, monotone, and
  transitive relation~$\rh$ between~$\Pfs(\set)$ and~$\set$.
\end{definition}

\begin{comment}
  By our terminology, we emphasise the feedback of algebra to logic
  while being faithful to Lorenzen. In a letter to
  Krull\footnote{Philosophisches Archiv, Universit{\"a}t Konstanz, PL
    1-1-131, published in \citealt[§~M]{neuwirthkonstanz}.} dated 13 March 1944, he writes: ``the insight that a
  \scentrel is intrinsically nothing more than a supersemilattice, and
  a valuation nothing more than a linear order [see
  \cref{sec:fore-integr-doma}], strikes me as the most essential
  result of my effort''.
  % Z.~B.~erscheint mir die Erkenntnis, daß ein Idealsystem eigentlich
  % nichts anderes als ein Ober\bs halbverband und eine Bewertung
  % nichts anderes als eine Ordnung ist, als das wesentlichste
  % Ergebnis meiner Bemühung.
  \eoe
\end{comment}

\begin{remark}
  \phantomsection
  \label{multiset}
  If instead of nonempty subsets, we had considered non\-empty
  multi\-sets, we would have had to add a contraction rule, and if we
  had considered nonempty lists, we would have had to add also a
  permutation rule.  \eoe
\end{remark}

Note the following banal generalisation of cut, using monotonicity: if
\(A\rh c\) and \(A',c\rh b\), with \(A'\) possibly empty, then
\(A,A'\rh b\).

\subsection{Fundamental theorem of \scentrels}
\label{sec:fund-theor-scentr}

A fundamental theorem holds for a \scentrel for a given set~$\set$: it
states that the relation generates a \indtr~$\slat$ whose order
reflects the relation. This is the single-conclusion analogue of the
better known \cref{thEntRelNB}.

%:    Theorem{thSemiEntRelNB}----------------
\begin{theorem}[fundamental theorem of \scentrels,
  {\citealp[see][Satz~3]{Lor1951}}]\label{thSemiEntRelNB}\hspace{-.5em}\footnote{\label{lorenzenlack}Our
    statement is the natural counterpart to Lorenzen's when using
    basic notions of universal algebra, and follows readily from his
    sketch of proof.}  Let \(\set\)~be a set and \(\rh\) a \scentrel
  for~\(\set\). Let us consider the \indtr~\(\slat\) defined by
  generators and relations in the following way: the generators are
  the elements of~\(\set\) and the relations are the
  \[
    \text{\(\Vii A\leq_\slat b\) whenever \(A\rh b\).}
  \]
  Then, for all~\((A,b)\) in~\(\Pfs(\set)\times \set\), we have the
  reflection of entailment
  \[
    \text{if \(\Vii A\leq_\slat b\), then \(A\rh b\).}
  \]
  In fact, \(\slat\)~can be defined as the ordered set obtained by
  descending to the quotient of \((\Pfs(\set),{\leq_{\rh}})\)
  by~\(=_{\rh}\), where \(\leq_{\rh}\)~is the meet-semilattice
  preorder defined by
  \begin{equation}
     A\leq_{\rh} B \enskip\equidef\enskip A\rh b\text{ for all \(b\in B\).}\label{defpreorder}
  \end{equation}
\end{theorem}
% --- end-theorem-----------------------------------------

\begin{proof}
  Let $A,B\in\Pfs(G)$: one has $\Vii A\leq_\slat\Vii B$ if and only if
  $\Vii A\leq_\slat b$ for all $b\in B$, i.e.\
  $ A\leq_{\rh} B$. The meet-semilattice~$\slat$ may therefore
  be generated in two steps.
  \begin{asparaenum}
  \item Let us check that $\leq_{\rh}$~is a preorder on~$\Pfs(G)$ that
    is compatible with the idempotent, commutative, and associative
    law of set union. Reflexivity of~$\leq_{\rh}$ follows from
    \cref{S0,S1}. Transitivity of~$\leq_{\rh}$ follows from~\cref{S1}
    and a repeated application of~\cref{S2}: if $A\rh c$ for
    every~$c\in C$ and $C\rh b$, then one may cut successively
    the~$c\in C$ and obtain $A\rh b$. Compatibility means that if
    $A\leq_{\rh}B$ and $A'\leq_{\rh}B'$, then $A,A'\leq_{\rh}B,B'$:
    this follows from~\cref{S1}.
  \item We may therefore define~\(\slat\) as the quotient of
    \((\Pfs(\set),{\leq_{\rh}})\) by~\(=_{\rh}\), with law~$\vii_S$
    obtained by descending the law of set union to the
    quotient.\qedhere
  \end{asparaenum}
\end{proof}

Note that the preorder $a\rh b$ on~$\set$ makes
its quotient a  subobject of~$\slat$  in the category of ordered sets.
\begin{remark}
  \phantomsection\label{dmn} The relation~$a\rh b$ is a priori just a
  preorder relation for~$\set$, not an order relation. Let us denote
  the element~$a$ viewed in the ordered set~$\overline{\set}$
  associated to this preorder by~$\overline a$, and
  let~$\overline A=\sotq{\overline a}{a\in A}$ for a subset~$A$
  of~$\set$.  In \cref{thSemiEntRelNB}, we construct a \indtr~$\slat$
  endowed with an order~$\leq_\slat$ that, loosely said, coincides with~$\rh$
  on~$\Pfs(\set)\times \set$; for the sake of rigour, we should have
  written above $\Vii\overline A\leq_\slat\overline b$ rather than
  $\Vii A\leq_\slat b$ in order to deal with the fact that the
  equality of~$\slat$ is coarser than the equality of~$\set$. In
  particular, it is~$\overline{\set}$ rather than~$\set$ which can be
  identified with a subset of~$\slat$.\eoe
\end{remark}

\begin{definition}
  The \scentrel $\rh[2]$ is \emph{coarser} than the \scentrel $\rh[1]$
  if $ A\rh[1]y$ implies \hbox{$ A\rh[2]y$}. One says also that
  $\rh[1]$ is \emph{finer} than~$\rh[2]$.
\end{definition}
This terminology has the following explanation: to say that the
relation~$\rh[2]$ is coarser than the relation~$\rh[1]$ is to say this
for the associated preorders, i.e.\ that $ A\leq_{\rh[1]}  B$
implies $ A\leq_{\rh[2]} B$, and this corresponds to the usual
meaning of \gui{coarser than} for preorders, since
$ A=_{\rh[1]} B$ implies accordingly $ A=_{\rh[2]} B$,
i.e.\ the equivalence relation~$=_{\rh[2]}$ is coarser
than~$=_{\rh[1]}$.

%%%%%%%%%%%%%%%%%%%%%%%%%%%%%%%%%%%%%%%%%%%%%%%%%%%%%%%%%%%%%%%%%%%%
%%%%%%%%%%%%%%%%%%%%%%%%%%%%%%%%%%%%%%%%%%%%%%%%%%%%%%%%%%%%%%%%%%%%

\subsection{\Sis}

Now suppose that $(G,{\leq_G})$~is an ordered monoid,\footnote{I.e.\ a
  monoid $(G,+,0)$ endowed with a (partial) order relation~$\leq_G$
  compatible with addition: $x\leq_G y\implies x+z\leq_G y+z$. We
  shall systematically omit the epithet ``partial''.} $(M,{\leq_M})$ a
\indtr-ordered monoid,\footnote{I.e.\ a monoid endowed with a \indtr
  law~$\vii$ inducing~$\leq_M$ and compatible with addition: the
  equality \( {x+{(y\vii z)}}={(x+y)}\vii{(x+z)} \) holds.} a
\emph{\mmonoid} for short, and $\mor\colon G\to M$ a morphism of
ordered monoids. The relation
\[
  a_1,\dots,a_k\rh b \enskip \equidef\enskip \mor(a_1)\vii\dots\vii\mor(a_k)\leq_M\mor(b)
\]
defines a \scentrel for~$G$ that satisfies furthermore the following
\properties:
\begin{alignat*}{2}
  \Rule{S3}&\text{if $a\leq_G b$, then }a\rh b&\text{\llap{(\presord);}}\\
  \Rule{S4}&\text{if $A\rh b$, then }x+A\rh x+b\quad(x\in G)&\qquad\text{(\equivariance).}
\end{alignat*}
\begin{definition}
  \label{defsysideals}
  An \emph{\si} for an ordered monoid~$G$ is a \scentrel~$\rh$ for~$G$
  satisfying \cref{S3,S4}.
\end{definition}
We propose to introduce \sis in a purely logical form, i.e.\ as
relations that require only a naive set theory for finitely enumerated
sets: this definition has been extracted from
\citealt[Definition~1]{Lor1939} \citep[compare][I, \S~3,
1]{Jaf1960}. One may also give them the form of predicates on~$\Pfs(G)$: see
\citealt[§~3]{coquandlombardineuwirthkonstanz}. The traditional form
of a meet-monoid for \sis may be recovered by \cref{ThSIJaf} below.

%r
%:     Remark{remJafsi}
\begin{remarks}
  \phantomsection
  \label{remJafsi}
  \begin{asparaenum}
  \item\label{remJafsi1} We find that it is more natural to state a
    direct implication rather than an equivalence in \cref{S3}; we
    deviate here from Lorenzen and Paul \citealt[page~16]{Jaf1960}. The
    reverse implication expresses the supplementary property that the
    \si is order-reflecting.
  \item\label{remJafsi2} \citet{Lor1939}, following at first Richard
    \citet{dedekind97} and Heinz \citet[\S~2]{Pru1932} in
    subordinating algebra to set theory, is describing a (finite)
    ``$r$-system'' of ideals through a set-theoretic map
    \[
      \Pfs(G)\longrightarrow \rP(G),\quad A \longmapsto \sotq{x\in G}{A\rh x}\eqdef A_r
    \]
    (here $\rP(G)$ stands for the set of all subsets of~$G$, and $r$~is just a variable name for distinguishing different systems) that satisfies the following properties:
    \begin{alignat*}{2}
      \Rule{I1}&A_r\supseteq A\text{;}&\\
      \Rule{I2}&A_r\supseteq B\enskip\implies\enskip A_r\supseteq B_r\text{;}&\\
      \Rule{I3}&{\so a}_r= \sotq{x\in G}{a\leq_G x}&\quad\text{(preservation and reflection of order);}\\
      \Rule{I4}&{(x+A)_r}=x+A_r&\text{(\equivariance).}
    \end{alignat*}
    This map has been called $'$-operation by \citet[Nr.~43]{Kru35}
    and is called star-operation today.  Let us note that the
    containment $A_r\supseteq B_r$ corresponds to the inequality
    $ A\leq_{\rh} B$ for the preorder associated with the
    \scentrel $\rh$ by the definition~\cref{defpreorder} above.
    As previously indicated, in contradistinction to Lorenzen and
    Jaffard, we find it more natural to relax the equality
    in \cref{I3} to a containment: if we do so, the reader can prove
    that the definition of star-operation is equivalent to
    \cref{defsysideals}; \cref{I1,I2} correspond to the definition of
    a \scentrel,\footnote{They can also be read as a finite version of
      Tarski's consequence operation (see \cref{beziau}).} and
    \namecrefs{I3}~\ref{I3} (relaxed) and~\ref{I4} correspond to
    \cref{S3,S4} in \cref{defsysideals}; compare
    \citealt[pages~504--505]{Lor1950}.

\item In the set-theoretic framework of the previous item, the $r_2$-system is coarser than the $r_1$-system exactly if $A_{r_2}\supseteq A_{r_1}$ holds for all~$A\in\Pfs(G)$ (see \citealp[page~509]{Lor1950}, and \citealp[I, \S~3, Proposition~2]{Jaf1960}).\eoe
\end{asparaenum}
\end{remarks}
%----------- fin remark ----------------------------------
\begin{comment}
  Lorenzen unveils the lattice theory hiding behind multiplicative
  ideal theory step by step, the decisive one being dated back by him
  to 1940. In a footnote to his definition, \citet[page~536]{Lor1939}
  writes: ``If one understood hence by a \scentrel every lattice that
  contains the principal ideals and satisfies
  \namecref{I4}~[\ref{I4}], then this definition would be only
  unessentially more comprehensive'' (it seems that Lorenzen is
  lacking the concept of \emph{semi}lattice at this stage of his
  research). \citet[page~486]{Lor1950} emphasises the transparency of
  this presentation as compared to the set-theoretic ideals: ``But if
  one removes this set-theoretic clothing, then the concept of ideal
  may be defined quite simply: a \scentrel of a preordered set is
  nothing other than an embedding into a semilattice.'' \eoe
\end{comment}

\subsection{The \mmonoid generated by an \si}
\label{sec:meet-semil-order}

The effectiveness of \cref{defsysideals} is shown by the following straightforward theorem, which boils down to acknowledging that the meet operation of set union on the preordered meet-semilattice $({\Pfs(G)},{\leq_{\rh}})$, described in the proof of \cref{thSemiEntRelNB}, is compatible with the monoid operation of set addition $A+B$.

%t
%:     Theorem{ThSIJaf}
\begin{theorem}
  \label{ThSIJaf}
  Let \(\rh\) be an \si for an ordered monoid~\(G\). Let \(\slat\)~be
  the \indtr generated by the \scentrel~\(\rh\). Then there is a
  (unique) monoid law on~\(\slat\) which is compatible with its
  semilattice structure and such that the natural morphism (of ordered
  sets) \(G\to \slat\) is a monoid morphism. The resulting
  meet-monoid~\(\slat\) is called the \emph{monoid of ideals}
  associated with~\(\rh\).
\end{theorem}

\begin{proof}\label{proofThSIJaf}
  We define $A+B=\sotq{a+b}{a\in A,\, b\in B}$ in $\Pfs(G)$. We have to
  check that this law descends to the quotient~$\slat$. It suffices to
  show that $ B\leq _{\rh} C$ implies
  $ A+B\leq _{\rh} A+C$. In fact, $ B\leq _{\rh} C$
  implies $ x+B\leq_{\rh} x+C$ by \equivariance, and
  ${ A+B}\leq_{\rh} x+C$ for every $x\in A$ by
  monotonicity. Finally, let us verify the compatibility of
  $\vii_{\slat}$ with addition: we note that already in $\Pfs(G)$, set
  union is compatible with set addition, i.e.\ ${A+(B,C)}={A+B,A+C}$.
\end{proof}

\subsection{The finest \si for an ordered monoid}\label{secSysId}

The finest \si admits the following description.

\begin{proposition}[{\citealp[Satz~14]{Lor1950}}]
  \label{comment-s-system1}
  Let \((G,\leq_G)\)~be an ordered monoid.  The finest \si for~\(G\)
  is defined by
  \[
    A\rh[\!\mathrm s]b\enskip\equidef\enskip a\leq_Gb\text{ for some
      \(a\in A\).}
  \]
  Note that \(\rh[\!\mathrm s]\)~is order-reflecting:
  \(a\rh[\!\mathrm s]b\iff a\leq_Gb\). The associated monoid of
  ideals is the meet-monoid freely generated by~\((G,\leq_G)\) (in the
  sense of the left adjoint functor of the forgetful functor).
\end{proposition}

\begin{proof}
  Left to the reader.
\end{proof}

\subsection{The system of Dedekind ideals}
\label{sec:fore-integr-doma}

Lorenzen's goal is to unveil the constructive content of Krull's
\foreignlanguage{german}{Fundamentalsatz}, i.e.\ to express it without
reference to valuations. In order to do so, consider an integral
domain~\(\id\) and its divisibility group~\(G=K\eti/\id\eti\) ordered
by divisibility, where \(K\) is the field of fractions of~\(\id\). A
valuation is a linear preorder~$\lt$ on~$G$ such that $1\lt x$
for~$x\in\id\etl$ and
\begin{equation}
  \label{valuation}
  \min(a_1,a_2)\lt a_1+a_2\quad\text{if $a_1+a_2\neq0$.}
\end{equation}
Property~\cref{valuation} implies that
$\min(a_1,\dots,a_k)\lt x_1a_1+\dots+x_ka_k$ if
$x_1a_1+\dots+x_ka_k\neq0$, where $x_1,\dots,x_k\in\id$. Let us write
$\gen{A}_\id$ for these linear combinations, where
$A=\{a_1,\dots,a_k\}$: we have $\gen A_\id\ni b\implies\min A\lt
b$. This motivates the following definition and observation.
\begin{definition}
  \label{dedekind}
  Let \(\id\) be an integral domain, \(K\) its field of fractions, and
  \(G=K\eti/\id\eti\) its divisibility group ordered by
  divisibility. The \emph{system of Dedekind ideals} for~\(G\) is
  defined by
  \[A\rh[\mathrm{d}]b\enskip\equidef\enskip \gen{A}_\id\ni b \text,\]
  where \(\gen{A}_\id\) is the fractional ideal generated by~\(A\)
  over~\(\id\) in~\(K\): if $ a_1,\dots,a_k$ are the elements of~$A$,
  then $\gen{A}_\id= \id\,a_1 + \cdots + \id\,a_k$.
\end{definition}

\begin{proposition}
  The system of Dedekind ideals for the divisibility group~\(G\) of an
  integral domain is an \si for~\(G\).
\end{proposition}

The above argument shows that a valuation may be defined as a linear
preorder that is coarser than the system of Dedekind ideals, so that
it gives rise to a homomorphism from the preordered meet-monoid of
Dedekind ideals into a linearly preordered group. In a letter to Krull
dated 6 June 1944,\footnote{Philosophisches Archiv, Universität
  Konstanz, PL~1-1-133, published in \citealt[§~R]{neuwirthkonstanz}.}
Lorenzen writes: ``If e.g.\ I replace the concept of valuation by
`homomorphism of a semilattice into a linearly preordered set', then I
see therein a conceptual simplification and not a complication. For
the introduction of the concept of valuation (e.g.\ the at first
arbitrary triangular inequality [\cref{valuation}]) is only justified
by the subsequent success, whereas the concept of homomorphism bears
its justification in itself. I would say that the homomorphism into a
linear preorder is the `pure concept' that underlies the concept of
valuation.''

\subsection{Forcing the positivity of an element}
\label{sec:forc-posit-an}

\begin{definition}
  Let $\rh$~be an \si for an ordered monoid~$G$ and $x\in G$. The
  system~$\rh[\!x]$ is the \si coarser than $\rh$ obtained by forcing
  the \property~$0\rh x$.
\end{definition}
The precise description of $\rh[\!x]$ given in the
\lcnamecref{lemrhgamma} below is the counterpart for a \scentrel to
the submonoid generated by adding an element~$x$ to a submonoid in an
ordered monoid (the ``$r$-extension'' $\dot{\mathfrak g}(x)_r$ of the
submonoid~$\dot{\mathfrak g}$, \citealp[page~516]{Lor1950}).

%l
%:     Lemma{lemrhgamma}
\begin{proposition}
  \label{lemrhgamma}
  Let \(\rh\)~be an \si for an ordered monoid~\(G\) and \(x\in G\).
  We have the equivalence
  \[A\rh[\!x]b\enskip\iff\enskip
    \text{there is \(p\geq 0\) such that
      \(A,A+x,\dots,A+px \rh b\).}
  \]
\end{proposition}
%----------- fin lemma -----------------------------------

Unlike the case of regular entailment relations (see \cref{A+px}), it
is not possible to omit $A+x,\dots,A+(p-1)x$ and to keep only $A,A+px$
to the left of~$\rh$: this can be seen in the
equivalence~\cref{lemrhgammadedekind} on
\cpageref{lemrhgammadedekind}, which is the application
of~\cref{lemrhgamma} to Dedekind ideals; see \citealt[Examples~8.1 and~8.2]{coquandlombardineuwirthkonstanz}.

\begin{proof}
  Let us denote by $A\rh'b$ the right-hand side in the equivalence
  above. In any \mmonoid, $0\leq x$ implies
  $\Vii(A,A+x,\dots,A+px) = \Vii A$, so that $A\rh'b$ implies
  $A\mathrel{\widetilde{\rh}}b$ for any \si $\widetilde\rh$ coarser
  than $\rh$ and satisfying $0\mathrel{\widetilde{\rh}}x$.

  It remains to prove that $A\rh'b$ defines an \si for $G$ (clearly
  $0\rh' x$ and $\rh'$ is coarser than $\rh$).  Reflexivity, \presord,
  \equivariance, and monotonicity are straightforward.  It remains to
  prove transitivity.  Assume that $A\rh' c$ and $A,c\rh' b$. We have
  to show that $A\rh' b$.  E.g.\ we have
  \begin{align}
    A,A+x,A+2x,A+3x &\rh c\text,\label{eqstarlemrhgamma}\\
    A,A+x,A+2x,c,c+x,c+2x &\rh b\text.\label{eqdaglemrhgamma}
  \end{align}
  \cref{eqstarlemrhgamma} gives by equivariance
  $A+2x,A+3x,A+4x,A+5x \rh c+2x$. By a cut with \cref{eqdaglemrhgamma}
  we may cancel out $c+2x$ and get
  \begin{equation}
    A,A+x,A+2x,A+3x,A+4x,A+5x,c,c+x \rh b.\label{eqdagdaglemrhgamma}
  \end{equation}
  The same argument allows us to cancel successively $c+x$ and $c$ out
  of~\cref{eqdagdaglemrhgamma}.
\end{proof}

\subsection{Forcing an element to be positive w.r.t.\ the system of
  Dedekind ideals}\label{sec:forc-posit-an-1}

\begin{proposition}
  \label{forcing-dedekind}
  Let \(\id\) be an integral domain, \(K\) its field of fractions and
  \(G=K\eti/\id\eti\) its divisibility group. Let \(x\in G\). Then the
  system~\(\mathrel{{(\rhd_{\mathrm{d}})}_{x}}\) obtained from the
  system of Dedekind ideals~\(\rh[\mathrm{d}]\) for~\(G\) by forcing
  \(1\rh[\mathrm{d}]x\) is the system of Dedekind ideals for the
  divisibility group of the extension~\(\id[x]\) of~\(R\) by a
  representative of~\(x\) in~\(K\).
\end{proposition}

\begin{proof}
  Forcing $1\rh[\mathrm{d}]x$ for an $x\in G$ amounts to
  replacing~$\id$ by~$\id[x]$ since \cref{lemrhgamma} tells that the
  resulting \si satisfies
  \begin{equation}
    A \mathrel{{(\rhd_{\mathrm{d}})}_{x}}b\enskip \iff\enskip
    \begin{aligned}[t]
      &\text{there is $p\geq 0$ such that}\\
      &A,Ax,\dots, Ax^p \rh[\mathrm{d}] b\text{ holds,}
    \end{aligned}
    \label{lemrhgammadedekind}
  \end{equation}
  which means that $ \gen A_{\id[x]}\ni b $.
\end{proof}

This is explained in \citealt[§~3]{Lor1953}, and has suggested \cref{lemrhgamma} to us.

\section{Lattice-ordered groups and regular \entrels}
\label{sec:grls-regular-entrels}

\subsection{\Trdis and \entrels}
\label{sec:trdis-entrels}

Let us define a \emph{\trdi} as a purely equational algebraic
structure with two laws $\vii$~and~$\vuu$ satisfying the axioms of
\trdis; we are leaving out the two axioms providing a greatest and a
least element.

For a \trdi~$\lat$, let us denote by $A\vda B$ the relation defined on
the set~$\Pfs(\lat)$ in the following way
\citep[see][Satz~5]{Lor1951}:
\[
  A \vda B \enskip \equidef \enskip \ndsp\Vii A\leq_\lat \Vuu B\text.
\]
This relation is reflexive, monotone, and transitive in the following
sense, expressed without the laws $\vii$~and~$\vuu$:
\begin{alignat*}{3}
  \Rule{R0} && a  &\vda a    &\text{(reflexivity);}     \\
  \Rule{R1} &\text{if $A \vda B$, then }& A,A' &\vda B,B'&\qquad\text{(monotonicity);}\\
  \Rule{R2} \text{if $A \vda B,c$ and }&A,c \vda B\text{, then }& A &\vda B &\text{(transitivity);}
\end{alignat*}
we insist on the fact that $A$~and~$B$ must be nonempty.

Note the following banal generalisation of cut, using monotonicity: if
\(A\vda B',x\) and \(A',x\vda B\), with \(A'\) and \(B'\) possibly
empty, then \(A,A'\vda B,B'\).

The following definition is a slight variant of a notion whose name
has been coined by Dana \citet[page~417]{Sco1974}. It is introduced as
description of a distributive lattice (see \cref{thEntRelNB}) in
\citealt[\S~2]{Lor1951}.

% --- Definition{defEntrelbis}-------------
\begin{definition}
  \phantomsection\label{defEntrelbis}
  Let $\set$~be an arbitrary set.
  \begin{asparaenum}
  \item A binary relation~$\vda$ on $\Pfs(\set)$ which is reflexive,
    monotone, and transitive is called an \emph{unbounded entailment
      relation}.
  \item The unbounded entailment relation~$\vda[2]$ is \emph{coarser}
    than the unbounded entailment relation~$\vda[1]$ if $A\vda[1]B$
    implies $A\vda[2]B$. One says also that $\vda[1]$ is \emph{finer}
    than~$\vda[2]$.
  \end{asparaenum}
\end{definition}
% --- end-definition------------------------------------

\Cref{multiset} applies again verbatim for \cref{defEntrelbis}.

\subsection{Fundamental theorem of unbounded entailment relations}

The counterpart to \cref{thSemiEntRelNB} for unbounded \entrels is
\cref{thEntRelNB}, a slight variant of the fundamental theorem of
entailment relations \citep[Theorem~1, obtained
independently]{CeCo2000}, which may in fact be traced back to
\citet[Satz~7]{Lor1951}.  It states that an unbounded \entrel for a
set~$\set$ generates a \trdi $\lat$ whose order reflects the relation.
The proof is the same as in \citealt{CeCo2000} or in \citealt[Theorem
XI-5.3]{CACM}.

% :    Theorem{thEntRelNB}----------------
\begin{theorem}[fundamental theorem of unbounded entailment relations,
  see
  {\citealp[Satz~7]{Lor1951}}]\hspace{-.5em}\footnote{\Cref{lorenzenlack}
    applies verbatim. Lorenzen's Satz~7 yields directly that if for
    every distributive lattice~$L$ and every $f\colon G\to L$ with
    $X\vda Y\implies \Vii f(X)\leq_L\Vuu f(Y)$ one has
    $\Vii f(A)\leq_L\Vuu f(B)$, then $A\vda B$. This may be considered
    as a result of completeness for the semantics of distributive
    lattices.}
  \label{thEntRelNB}
  Let \(\set\)~be a set and \(\vda\)~an unbounded entailment relation
  on~\(\Pfs(\set)\). Let us consider the \trdi~\(\lat\) defined by
  generators and relations in the following way: the generators are
  the elements of~\(\set\) and the relations are the
  \[
    \text{\(\Vii A\leq_\lat\Vuu B\) whenever \(A\vda B\).}
  \]
  Then, for all~\(A\), \(B\) in~\(\Pfs(\set)\), we have the reflection of entailment
  \[
    \text{if \(\Vii A\leq_\lat\Vuu B\), then \(A\vda B\).}
  \]
\end{theorem}
% --- end-theorem-----------------------------------------

\Cref{dmn} applies again mutatis mutandis.

\subsection{Regular \entrels}

Let $(G,\leq_G)$~be an ordered monoid, $(H,{\leq_H})$~a distributive
lattice-ordered monoid,\footnote{I.e.\ a meet-monoid endowed with a
  join-semilattice law~$\vuu$ inducing~$\leq_H$ that is distributive
  over~$\vii$ and compatible with addition.} and $\mor\colon G\to H$ a
morphism of ordered monoids. The laws $\vii$~and~$\vuu$ on~$H$ provide
a \trdi structure, and the relation
\[
  a_1,\dots,a_k\vda b_1,\dots,b_\ell \enskip \equidef\enskip \mor(a_1)\vii\dots\vii\mor(a_k)\leq_H\mor(b_1)\vuu\dots\vuu\mor(b_\ell)
\]
defines an unbounded entailment relation for~$G$ that satisfies
furthermore the following straightforward \properties:
\begin{alignat*}{2}
  \Rule{R3}\quad&\text{if $a\leq_G b$, then }a\vda b&\text{\llap{(\presord);}}\\
  \Rule{R4}\quad&\text{if $A\vda B$, then }x+A\vda x+B\quad(x\in G)&\qquad\text{(\equivariance).}
\end{alignat*}

Now suppose that $(H,{\leq_H})$ is a lattice-ordered
group,\footnote{An \emph{ordered group} is a group that is an ordered
  monoid. If it is meet-semilattice-ordered, then it turns out that it
  is a \emph{lattice-ordered group} with join defined by
  $a\vuu b=-(-a\vii-b)$.} an \emph{\grl} for short. Then the following
further \property holds:
\begin{alignat*}{2}
  \Rule{R5}\quad&x+a,y+b\vda y+a,x+b&\qquad\text{(regularity).}
\end{alignat*}
This follows from the observation that if $x',a',y',b'$ are elements
of~$H$, then the difference of right-hand side and left-hand side of
\begin{equation}
  {(x'+a')\vii(y'+b')}\leq_H{(y'+a')\vuu(x'+b')}\label[inequality]{reginlg}
\end{equation}
is
\[
  \begin{multlined}
    \bigl((y'+a')\vuu(x'+b')\bigr)+\bigl((-x'-a')\vuu(-y'-b')\bigr)\\
    \begin{aligned}
      &=_H(y'-x')\vuu(a'-b')\vuu(b'-a')\vuu(x'-y')\\\label{reg-lgroup}
      &=_H|y'-x'|\vuu|b'-a'|\text.
    \end{aligned}
  \end{multlined}
\]
We assemble these observations into the following new purely logical
definitions \citep[compare][\S~1]{Lor1953}, given for ordered monoids
even though we study them only in the case of ordered groups.
\begin{definition}
  \label{defregsysideals}
  Let $G$~be an ordered monoid.
  \begin{asparaenum}
    % :  definition  defuppersysideals
  \item\label{defregsysideals-equivariant} An \emph{\equivariant}
    \entrel for~$G$ is an unbounded entailment relation~$\vda$ for~$G$
    satisfying \cref{R3,R4}.
    % : definition defregsysideals
  \item A \emph{regular entailment relation} for~$G$ is an
    \equivariant \entrel for~$G$ satisfying \cref{R5}.
  \item An \si for~$G$ is \emph{regular} if it is the restriction of a
    regular \entrel to~$\Pfs(G)\times G$.
  \end{asparaenum}
\end{definition}
We prefer the terminology in \cref{defregsysideals-equivariant} to
Lorenzen's vocable ``upper system of ideals''; note that a
fundamental theorem is also available for this concept, but we shall
not need it. A key fact to be established is that a regular entailment
relation is determined by its restriction to~$\Pfs(G)\times G$ (see
\cref{ABA-B} of \cref{cor+x-x}). This allows one to give it the form of a predicate on~$\Pfs(G)$: see
\citealt[§~2]{coquandlombardineuwirthkonstanz}. 

\begin{comment}
  Lorenzen discovers the property of regularity in his analysis of the
  case of noncommutative groups: he isolates
  \cref{reginlg}, which is trivially verified in a
  commutative \grl, but not in a noncommutative
  one. \citet[Satz~13]{Lor1950} proves by a well-ordering argument
  that a (noncommutative) preordered $\ell$-group satisfying \cref{reginlg} is a subdirect product
  of linearly preordered groups. In the commutative setting, this
  corresponds to the theorem (in classical mathematics) that
  any commutative preordered $\ell$-group is a subdirect product of linearly
  preordered commutative groups.\eoe
\end{comment}

\subsection{Regularity as the right to assume elements linearly
  ordered}

Let us now undertake an investigation of regular entailment relations
as defined in \cref{defregsysideals}.

\begin{lemma}\label{lem1thGOEntrelGRL}
  Let \(\vda\)~be an unbounded entailment relation for an ordered
  group~\(G\). \Cref{R5} may be restated as follows:
  \[
    \text{if \(x_1+x_2=_G y_1+y_2\), then \(x_1,x_2\vda y_1,y_2\).}
  \]
\end{lemma}

\begin{proof}
  By \cref{R5}, $y_1+(y_2-x_2),x_2\vda y_1,{x_2+(y_2-x_2)}$, and if
  $x_1+x_2=_Gy_1+y_2$, then $y_1+(y_2-x_2)=_Gx_1$.
\end{proof}

In the remainder of this section, $\vda$ is a regular entailment
relation for an ordered group~$G$.

\begin{lemma}\label{lem2thGOEntrelGRL}
  Let \(A\in\Pfs(G)\) and \(x\in G\). In the distributive lattice~\(L\) generated by~\(\vda\) (\cref{thEntRelNB}), \(\Vii A\leq_L(\Vii A+x)\vuu(\Vii A-x)\) holds.
\end{lemma}

\begin{proof}
  For every $a,a'\in A$, $a,a'\vda a+x,a'-x$ holds by
  \cref{lem1thGOEntrelGRL}. Therefore
  \[
    \Vii A\leq_L\Vii_{a,a'\in A}(a+x\vuu a'-x)=_L\Bigl(\Vii_{a\in
      A}a+x\Bigr)\vuu\Bigl(\Vii_{a'\in A}a'-x\Bigr)\text.\qedhere
  \]
\end{proof}

\begin{lemma}\label{A+xA-x}
  Let \(A,B\in\Pfs(G)\) and \(x\in G\).
  \begin{asparaenum}
  \item\label{ceun} If \(A,A+x\vda B\) and \(A,A-x\vda B\), then
    \(A\vda B\).
  \item\label{cedeux} \(A,A+x\vda B\) holds if and only if
    \(A\vda B,B-x\).
  \end{asparaenum}
\end{lemma}

\begin{proof}
  \begin{asparaenum}
  \item \Cref{thEntRelNB} allows us to work in the distributive lattice~\(L\) generated
    by~\(\vda\).  We have $\Vii A\vii\bigl(\Vii A+x\bigr)\leq_L\Vuu B$ and
    $\Vii A\vii\bigl(\Vii A-x\bigr)\leq_L\Vuu B$. By
    \cref{lem2thGOEntrelGRL},
    \[
      \begin{aligned}
        \Vii A&=_L\Vii A\vii\Bigl(\Bigl(\Vii A+x\Bigr)\vuu\Bigl(\Vii A-x\Bigr)\Bigr)\\
        &=_L\Bigl(\Vii A\vii\Bigl(\Vii A+x\Bigr)\Bigr)\vuu\Bigl(\Vii A\vii\Bigl(\Vii A-x\Bigr)\Bigr)\leq_L\Vuu B\text.
      \end{aligned}
    \]
  \item Suppose that $A,A+x\vda B$. Then $A-x,A\vda B-x$, so that
    $A,A+x\vda B,B-x$ and $A,A-x\vda B,B-x$, and we may apply
    \cref{ceun}. The converse holds because the relation converse
    to~$\vda$ is a regular \entrel for $(G,\geq_G)$.\qedhere
\end{asparaenum}
\end{proof}

\begin{lemma}\label{aa+qx}
  Let \(a,x\in G\) and \(0\leq p\leq q\). Then \(a,a+qx\vda a+px\).
\end{lemma}

\begin{proof}
  By induction on~$q$. This is trivial if $p=0$ or $p=q$. Suppose that
  \(a,a+q'x\vda a+px\) whenever \(0\leq p\leq q'<q\). Consider first
  $1\leq p\leq q-p$. By hypothesis, $a,a+(q-p)x\vda a+px$. By
  regularity, $a,a+qx\vda a+px,a+(q-p)x$. A cut yields
  $a,a+qx\vda a+px$. Make now an induction on~$p$ with $q-p\leq
  p<q$. Suppose that \(a,a+qx\vda a+p'x\) for $0\leq p'<p$. As
  $0\leq 2p-q<p$, we have $a,a+qx\vda a+(2p-q)x$. By regularity (and contraction),
  $a+qx,a+(2p-q)x\vda a+px$. A cut yields $a,a+qx\vda a+px$.
\end{proof}

\begin{lemma}
  \label{A+px}
  Let \(A,B\in\Pfs(G)\) and \(x\in G\). Let $0\leq p\leq q$.  If
  \(A,A+px\vda B\) holds, or merely \(A,A+px,A+qx\vda B\), then so
  does \(A,A+qx\vda B\).
\end{lemma}

\begin{proof}
  Cut successively the $a+px$ for~$a\in A$ in the given entailment
  with the entailment \(a,a+qx\vda a+px\) holding by \cref{aa+qx}.
\end{proof}

Let us now give a description of the regular entailment relation
obtained by forcing an element~$x$ to be positive.

\begin{proposition}\label{forcingreg}
  Let \(\vda\)~be a regular entailment relation for an ordered
  group~\(G\). Let us define the relation~\(\vda[x]\) on~$\Pfs(G)$ by
  writing \(A\vda[x]B\) if there is $p\geq0$ such that
  \(A,A+px\vda B\). Then \(\vda[x]\) is a regular entailment relation,
  and it is the finest equivariant entailment relation~\(\vda'\)
  coarser than~\(\vda\) such that \(0\vda'x\).
\end{proposition}

\begin{proof}
  Only transitivity needs an argument. Suppose that $A,A+px\vda B,c$
  and $A,c,A+qx,c+qx\vda B$. By \cref{A+px}, we may suppose $p=q$; let
  $y=px=qx$. By equivariance, $A+y,A+2y\vda B+y,c+y$. Let us consider
  $A'=A,A+y,A+2y$ and prove $A'\vda B$. By monotonicity, $A'\vda B,c$
  and $A',c,c+y\vda B$ and $A'\vda B+y,c+y$. The two last yield by a
  cut $A',c\vda B,B+y$, which by \cref{cedeux} of \cref{A+xA-x} yields
  $A',c,A'-y,c-y\vda B$. But monotonicity also yields
  $A',c,A'+y,c+y\vda B$, so that by \cref{ceun} of \cref{A+xA-x}
  follows $A',c\vda B$. A cut yields $A'\vda B$. \Cref{A+px} produces
  $A,A+2y\vda B$, and therefore $A\vda[x]B$.

  The relation~$\vda[x]$ is clearly coarser than~$\vda$ and satisfies
  \(0\vda[x]x\). Conversely, suppose that $A\vda[x]B$, i.e.\
  $A,A+px\vda B$ for some $p\geq0$, and consider an equivariant
  entailment relation~$\vda'$ coarser than~$\vda$ and satisfying
  $0\vda'x$. Then $A,A+px\vda'B$ and, because
  $a\vda'a+x\vda\cdots\vda'a+px$ for each~$a\in A$, we may cut
  successively the~$a+px$ and obtain $A\vda'B$.
\end{proof}

\begin{theorem}\label{+x-x}
  Let \(\vda\)~be a regular entailment relation for an ordered
  group~\(G\). If \(A\vda[x]B\) and \(A\vda[-x]B\), then \(A\vda B\).
\end{theorem}

\begin{proof}
  By \cref{forcingreg}, $A,A+px\vda B$ and $A,A-qx\vda B$ for some
  $p,q\geq0$. By \cref{A+px}, we may suppose $p=q$, and conclude by
  \cref{ceun} of \cref{A+xA-x}.
\end{proof}

The meaning of \cref{+x-x} is that if one wants to establish an
entailment involving certain elements, one can always assume that
these elements are linearly ordered. \citet[Principle XI-2.10]{CACM}
call this the ``Principle of covering by quotients for \grls''.

\subsection{Consequences of assuming elements linearly ordered:
  cancellativity}
\label{sec:cancellativity}

If $A$~and~$B$ are linearly ordered nonempty finitely enumerated
subsets of~$G$, then
\[
  \begin{gathered}
    \min(A+B)\leq_G\min B\implies\min A\leq_G0\text,\\
    \min A\leq_G\max B\iff\min(A-B)\leq_G0\iff0\leq_G\max(B-A)\text.
  \end{gathered}
\]
We have therefore the following corollary to \cref{+x-x}.
\begin{corollary}\label{cor+x-x}
  Let \(\vda\) be a regular entailment relation for~\((G,\leq_G)\) and
  \(A,B\in\Pfs(G)\).
  \begin{asparaenum}
  \item\label{cancellative} If \(A+B\vda b\) for every~\(b\in B\),
    then \(A\vda0\).
  \item\label{ABA-B} \(A\vda B\) holds if and only if \(A-B\vda0\), if
    and only if \(0\vda B-A\).
  \end{asparaenum}
\end{corollary}

Let $\rh$ be the \scentrel defined as the restriction of~$\vda$ to
$\Pfs(G)\times G$.
\begin{asparaitem}
\item \Cref{cancellative} expresses that the meet-monoid associated
  with~$\rh$ is cancellative (see \cref{thGoembedsGrl3} of \cref{thGoembedsGrl}). In \cref{secSymInfmo}, we provide the
  construction of its Grothendieck group, which is an $\ell$-group,
  and draw the conclusion that the underlying distributive lattice
  coincides with the one generated by~$\vda$.
\item \Cref{ABA-B} expresses that $\vda$ is determined by
  $\rh$. Conversely, given a \scentrel~$\rh$, there are several
  unbounded \entrels that reflect~$\rh$: see \citealt[\S~3]{Lor1952}, and
  \citealt[\S~3.1]{rinaldischusterwessel16}.
\end{asparaitem}

%%%%%%%%%%%%%%%%%%%%%%%%%%%%%%%%%%%%%%%%%%%%%%%%%%%%%%%%%%%%%%%%%%%%
%%%%%%%%%%%%%%%%%%%%%%%%%%%%%%%%%%%%%%%%%%%%%%%%%%%%%%%%%%%%%%%%%%%%

\section{Consequences of cancellativity}\label{sec:cons-canc}

\subsection{The Grothendieck \grl of a \indtr-or\-de\-red
  monoid}\label{secSymInfmo}

\begin{definition}
  Let \((M,+,0)\) be a monoid. The \emph{Grothendieck group of~\(M\)}
  is the group freely generated by \((M,+,0)\) (in the sense of the
  left adjoint functor of the forgetful functor).
\end{definition}

\begin{lemma}\label{sec:groth-grl-indtr}
  The Grothendieck group~\(H\) of~\(M\) may be obtained by considering
  the monoid of formal differences \(a-b\) for \(a,b\in M\), equipped
  with the addition \((a-b)+(c-d)=(a+c)-(b+d)\) and the neutral
  element \(0-0\), and by taking its quotient by the equality
  \[a-b=_H c-d\enskip\equidef\enskip\exists x\in M\
    a+d+x=_Mb+c+x\text.\] Every equality $a-b=_H c-d$ may be reduced
  to two elementary ones, i.e.\ of the form $e-f=_H(e+y)-(f+y)$.
\end{lemma}

\begin{proof}
  See \citealt[I, §~2.4]{bourbaki74}, but for the last assertion, which
  follows from transitivity and symmetry of~$=_H$:
  \[
    a-b=_H (a+d+x)-(b+d+x)=_H(b+c+x)-(b+d+x)=_H c-d\text.\qedhere
  \]
\end{proof}

The following easy construction, for which we did not locate a good
reference \citep[but compare][§~2.4]{cignolidottavianomundici00}, is
particularly significant in the case where the \mmonoid associated
with an \si proves to be cancellative.
%t
%:     Theorem{thGoembedsGrl}
\begin{theorem}\label{thGoembedsGrl}
  Let \((M,+,0,\vii)\) be a \mmonoid.  Let \(H\) be the Grothendieck
  group of \(M\) with monoid morphism \(\mor\colon M\to H\).
  \begin{asparaenum}
  \item\label{thGoembedsGrl1} There is a unique \mmonoid structure
    on \(H\) such that \(\mor\) is a morphism of ordered sets.
  \item\label{thGoembedsGrl2} \((H,+,-,0,\vii)\) is an \grl: it is the
    \grl freely generated by \((M,+,0,\vii)\) (in the sense of the
    left adjoint functor of the forgetful functor), and called the
    \emph{Grothendieck \grl} of~\(M\).
  \item\label{thGoembedsGrl3} Assume that \(M\) is cancellative, i.e.\
    that \(a+x=_Mb+x\) implies \(a=_Mb\).  Then \(\mor\) is an
    embedding of \mmonoids.
  \end{asparaenum}
\end{theorem}
%----------- fin theorem -----------------------------

\begin{proof}
  \begin{asparaenum}
  \item[(\ref{thGoembedsGrl1})\enspace\Bullet] When trying to define
    $z=(e-f)\vii(i-j)$ we need to ensure that
    $f+j+z=_M(e+j)\vii (i+f)$: so we claim that
    $z\eqdef ((e+j)\vii (i+f))-(f+j)$ will do.  Let us show first that
    the law $\vii$ is well-defined on $H$: by
    \cref{sec:groth-grl-indtr}, it suffices to show that
    $z=_H{((e+y)-(f+y))}\vii{(i-j)}$, which reduces successively to
    ${((e+j)\vii (i+f))-(f+j)}=_H{((e+j+y)\vii(i+f+y))}-{(f+j+y)}$ and
    to $((e+j)\vii(i+f))+{(f+j+y)}=_M{((e+j+y)\vii(i+f+y))}+{(f+j)}$.
    Since $\vii$ is compatible with $+$ in $M$, both sides are equal
    to ${(e+2j+f+y)}\vii{(i+2f+j+y)}$.
  \item[\Bullet] The map $\mor\colon M\to H$ preserves $\vii$: in fact
    $\mor(a)\eqdef a-0$, and the checking is immediate.
  \item[\Bullet] The law $\vii$ on $H$ is idempotent, commutative, and
    associative.  This is easy to check and left to the reader.
  \item[\Bullet] The law $\vii$ is compatible with $+$ on $H$. This is
    easy to check and left to the reader.
  \item[(\ref{thGoembedsGrl2})] This construction yields the \grl
    freely generated by~\(M\) because it only uses the hypothesis that
    there is a meet-monoid morphism~$\mor\colon M\to H$.
  \item[(\ref{thGoembedsGrl3})] Cancellativity may be read precisely as
    the injectivity of~$\mor$. The \mmonoid structure is purely
    equational, so that an injective morphism is always an
    embedding.\qedhere
  \end{asparaenum}
\end{proof}

\subsection{The \grl generated by a regular entailment
  relation}\label{sec:proof-thref}

The main result of this article is \cref{ithGOEntrelGRL} below: it
states that regular entailment relations provide a description of all
morphisms from an ordered group~$G$ to \grls generated by (the image
of)~$G$.

% : theorem \label{ithGOEntrelGRL}
\begin{theorem}\label{ithGOEntrelGRL}
  Let \(\vda\)~be a regular entailment relation for an ordered
  group~\(G\).  Let \(H\) be the \trdi generated by the entailment
  relation~\(\vda\). Then there is a (unique) group law on~\(H\) which
  is compatible with its lattice structure and such that the natural
  morphism (of ordered sets)~\(G\to H\) is a group morphism. The
  resulting \grl is called the \emph{group of ideals} associated
  with~\(\vda\).
\end{theorem}

\begin{proof}
  By \cref{cancellative} of \cref{cor+x-x}, the meet-monoid associated
  with the restriction~$\rh$ of~$\vda$ is cancellative, so that by
  \cref{thGoembedsGrl3} of \cref{thGoembedsGrl} it embeds into its
  Grothendieck \grl~$H$. The underlying distributive lattice coincides
  with the one generated by~$\vda$ by \cref{ABA-B} of
  \cref{cor+x-x}. Uniqueness follows from \cref{thGoembedsGrl2} of
  \cref{thGoembedsGrl}.
\end{proof}

Let us state a variant of \cref{ithGOEntrelGRL}.

%c
%:     Corollary{corthGOEntrelGRL}
\begin{corollary}\label{corthGOEntrelGRL}
  Let \((G,{\leq_G})\)~be an ordered group and \(\rh\)~an \si
  for~\(G\). The following are equivalent:
  \begin{asparaenum}
  \item\label{corthGOEntrelGRL1} The \si~\(\rh\) is regular (i.e.\ it
    is the restriction of a regular entailment relation~\(\vda\)).
  \item\label{corthGOEntrelGRL2} The \mmonoid associated with the \si
    \(\rh\) for \(G\) (\cref{ThSIJaf}) is cancellative.
  \end{asparaenum}
\end{corollary}
% --------- fin corollary -------------------------------
% 
\begin{proof}
  \begin{asparaitem}
  \item[$(\ref{corthGOEntrelGRL1})\implies(\ref{corthGOEntrelGRL2})$.]
    The subset~$M\subseteq H$ of those elements that may be written
    $\mor(x_1)\vii\dots\vii\mor(x_n)$ for some $x_1,\dots,x_n\in G$ is
    the \indtr associated with the \si~$\rh$ obtained by
    restricting~$\vda$ to~$\Pfs(G)\times G$. This subset is stable by
    addition, so that the restriction of addition to $M$ endows it
    with the structure of a cancellative \mmonoid. Thus $H$ is
    necessarily (naturally isomorphic to) the Grothendieck \grl of
    $M$.
  \item[$(\ref{corthGOEntrelGRL2})\implies(\ref{corthGOEntrelGRL1})$.] See
    the proof of \cref{ithGOEntrelGRL}.  \qedhere
  \end{asparaitem}
\end{proof}

\begin{comment}\label{remhist2}
  \Cref{ithGOEntrelGRL} is new and replaces the second step of the
  proof of Satz~1 in \citealt{Lor1953} (see \cpageref{satz1}), which
  establishes that the distributive lattice~$H$ is in fact an \grl by
  constructing ``by hand'' a group law without emphasis on the rôle of
  regularity.  This rôle is revealed by our presentation, which allows
  for more conceptual arguments.\eoe
\end{comment}

\section{The regularisation of an \si for an ordered
  group}\label{sec:regularisation}

\subsection{Definition}

An \si gives rise to a regular entailment relation if one proceeds as if elements occurring in a computation are
comparable.  More precisely, this idea gives the following definition.
%:  definition ideflorrel
\begin{definition}[{see \citealp[(2.2) and page 23]{Lor1953}}]
  \label{ideflorrel} Let $\rh$~be an \si for an ordered group~$G$.
\begin{asparaenum}
\item For $y_1,\dots,y_n\in G$, consider the
  \si $\rh[\!y_1,\dots,y_n]$ coarser than $\rh$ obtained by forcing
  the \properties $0\rh y_1$, \dots,~$0\rh y_n$.  The
  \emph{regularisation} of~$\rh$ is the relation on~$\Pfs(G)$ defined
  by
  \begin{equation*}
    A\vda[\rh]B\enskip\equidef\enskip
    \begin{aligned}[t]
      &\text{there are $x_1,\dots,x_m\in G$ such that for every }\\
      &\text{choice of signs $\pm$, $A-B\rh[\pm x_1,\dots,\pm x_m]0$ holds.}
    \end{aligned}
  \end{equation*}
\item\label{ideflorrel1} An element~$b$ of $G$ is
  \emph{$\rh$-dependent on~$A$} if $A\vda[\rh]b$.
\item\label{ideflorrel2} The group~$G$ is \emph{$\rh$-closed} if
  $\vda[\rh]$ reflects the order on~$G$, i.e.\ if the implication
  ${a\vda[\rh]b}\implies{a\leq_Gb}$ holds for all $a,b\in G$.
\end{asparaenum}
\end{definition}
The terminology of \cref{ideflorrel1,ideflorrel2} comes from integral
domains.  Regularisation is an early occurrence of dynamical
algebra \citep[see][]{CLR2001}: we shall see that $\rh$-closedness is
the dynamical counterpart to being embedded into a product of linearly
preordered groups.  Let us go through a simple example that
illustrates a relevant feature of this construction (compare
\cref{lemmath2GOEntrelGRL0}).

%:     Example{exaideflorrel}
\begin{example}
  \label{exaideflorrel}
  Let us apply a case-by-case reasoning in order to prove that in a
  linearly ordered group, if $n_1a_1+\dots+n_ka_k\leq 0$ for some
  integers~$n_i\geq0$ not all zero, then $a_j\leq 0$ for some~$j$. If
  $a_j\leq 0$ for some~$j$, everything is all right. If $0\leq a_j$
  for all~$j$, take~$i$ such that $n_i\geq1$: then
  ${a_i}\leq{n_ia_i}\leq{n_1a_1+\dots+n_ka_k}\leq0$. The conclusion
  holds in each case.  Similarly, assume that
  $n_1a_1+\dots+n_ka_k\rh 0$ with $n_i\geq 0$ not all zero. We have
  $a_j\rh_{-a_j}0$ for each~$j$. By monotonicity,
  $a_1,\dots,a_k\rh[\epsilon_1 a_1,\dots,\epsilon_k a_k]0$ holds if at
  least one $\epsilon_j$ is equal to~$-1$. If we force $0\rh a_j$ for
  all~$j$, take~$i$ such that $n_i\geq1$: then
  ${a_i}\leq_{\rh}{n_ia_i}\leq_{\rh}{n_1a_1+\dots+n_ka_k}\leq_{\rh}0$. This
  proves that \(a_1,\dots,a_k\rh[+ a_1,\dots,+ a_k]0\). We conclude
  that $a_1,\dots,a_k\vda[\rh]0$.\eoe
\end{example}
% --------- fin example ----------------------------------------

\begin{comment}
  \label{remhist3}
  \citet[pages~488--489]{Lor1950} describes the basic idea of
  regularisation. He says that (the single-conclusion counterpart to)
  it is his answer to the following question for a given \si~$\rh$
  with associated monoid of ideals~$H_r$ (``$\mkern-1muH\mkern1mu$''
  for ``\foreignlanguage{german}{Halbverband}'', semi\-lattice, $r$~a
  variable name for distinguishing different monoids):
  \begin{quotation}
    How can the distinguished brauchbar [equivariant] \scentrel of a
    preordered group~$G$ be constructed---assumed that it actually
    exists?

    We thus assume that the preorder of~$G$ is representable as
    conjunction of allowable linear preorders of~$G$. But we do not
    want to use the knowledge of these linear preorders, because we
    precisely want to determine the allowable linear preorders by the
    aid of the \si.
  \end{quotation}

  As customary in the literature, we do not translate the epithet
  ``\foreignlanguage{german}{brauchbar}'' (which means ``usable'')
  introduced by Krull. Let us say that a linear preorder~$\lt$ on~$G$
  is \emph{$\rh$-allowable} if $A\rh b\implies\min A\lt b$; Lorenzen
  notes that it defines a ``linear'' \si~$\rh'$ coarser than~$\rh$:
  $\lt$~extends in a unique way to the linear preorder on the monoid
  of ideals given by $A\leq_{\rh'}B\equidef\min A\lt\min
  B$. Lorenzen's question is therefore about the \si~$\rh[\mathrm a]$
  defined by
  \[
    A\rh[\mathrm a] b\enskip\equidef\enskip
    \begin{aligned}[t]
      &\text{for every $\rh$-allowable linear preorder~$\lt$,}\\
      &\min A\lt b\text{ holds,}
    \end{aligned}
  \]
  and his answer is that $A\rh[\mathrm a]b\iff A\vda[\rh] b$
  \citep[Satz~24]{Lor1950}. More precisely, it is straightforward that
  every linear preorder coarser than~$\rh$ is also coarser than
  $\vda[\rh]$, so that $A\vda[\rh] b\implies A\rh[\mathrm
  a]b$. Conversely, consider $A$~and~$b$ such that $A\vda[\rh]b$ does
  not hold: a well-ordering argument grants a coarsest \si~$\rh'$
  coarser than~$\rh$ such that $A\vda[\rh']b$ still does not hold; it
  cannot be other than linear, for if both $0\rh'x$ and $x\rh'0$ did
  not hold for some~$x\in G$, then both $\rhd'_x$ and $\rhd'_{-x}$
  would be coarser than~$\rh'$, so that $A\vda[\rhd'_x]b$ and
  $A\vda[\rhd'_{-x}]b$ would hold, i.e.\ $A\vda[\rh']b$.

  In the literature, ``endlich arithmetisch brauchbar'' is the
  property of cancellativity of a monoid of ideals, to be found in
  several places of this article: \cref{cancellative} of
  \cref{cor+x-x} (w.r.t.\ regular \entrels); \cref{thGoembedsGrl3} of
  \cref{thGoembedsGrl} (w.r.t.\ the Grothendieck \grl);
  \cref{corthGOEntrelGRL} (equivalence with regularity);
  \cref{corthmDivLor} (Macaulay's theorem); \cref{lemSIJaf} (Prüfer's
  property~$\Gamma$); \cref{thSIPrufer} (Prüfer's theorem). In our
  presentation, cancellativity of the regularisation is a key feature.

  In the case of integral domains, the allowable linear preorders are
  exactly the valuations, and the \si~$\rh[\mathrm{a}]$ is designed so
  as to capture the characterisation of integral dependence of an
  element~$b$ on the ideal generated by a finitely enumerated set~$A$
  through valuations (the letter ``a'' has been chosen by
  \citealp{Pru1932} for ``algebraically representable'').\eoe
\end{comment}

Note that \cref{subseclcdalaLor,subsecaritalaLor} below do not resort
to the fundamental theorem of unbounded entailment relations,
\cref{thEntRelNB}, i.e.\ the reasoning takes place on the level of the
entailment relations and not in the generated \trdi.

\subsection{The regularisation is the regular \entrel generated by an \si}\label{sec:proof-ithmDivLorsi}

\citet[\S~2]{Lor1953} starts with an \si~$\rh$ for an ordered group
and uses the heuristics of \cref{ABA-B} of \cref{cor+x-x} to define
the regularisation~$\vda[\rh]$ as in \cref{ideflorrel}. Then he
applies the fundamental theorem for unbounded entailment relations,
\cref{thEntRelNB}, and obtains a distributive lattice. This article
wishes to assess the following remarkable theorem (which holds also
for noncommutative groups).

\begin{satzeins}[{\citealt{Lor1953}}]\label{satz1}
  Let \(\rh\)~be an \si for an ordered group~\(G\). Its
  regularisation~\(\vda[\rh]\) is a regular entailment relation and
  the action of~\(G\) on the distributive lattice~\(H\) generated
  by~\(\vda[\rh]\) may be extended to a group law for~\(H\).
\end{satzeins}

\Cref{ithmDivLorsi} below strengthens the first step of the
proof of Satz~1, in which the entailment relation~$\vda[\rh]$ is
constructed and shown to be regular as in \cref{lem1ithmDivLorsi}.  In
our analysis of Lorenzen's proof, we separate this step from its
second step, the explicit construction of a group law for the
regularisation. Our presentation makes regularity (\cref{R5}) the
lever for sending~$G$ homomorphically into an \grl in
\cref{ithGOEntrelGRL}.

\begin{proposition}\label{lem1ithmDivLorsi}
  Let \(\rh\) be an \si for an ordered group~\(G\). Its
  regularisation~\(\vda[\rh]\) is a regular entailment relation
  for~\(G\).
\end{proposition}

\begin{proof}
  The regularisation is clearly reflexive and monotone, and satisfies
  \cref{R3,R4}.

  \emph{Let us prove that the regularisation is transitive.} Suppose
  that ${A,0}\vda[\rh]B$ and $A\vda[\rh]{0,B}$: there are
  $x_1,\dots,x_m,y_1,\allowbreak\dots,y_n$ such that for every choice
  of signs~$\pm$,
  ${A-B,-B}\rh[\!\pm x_1,\dots,\pm x_m]\!0$
   and
  $A,A-B\rh[\!\pm y_1,\dots,\pm y_n]\!0$ hold.

  Let $A=\so{a_1,\dots,a_k}$. If $a_i\rh 0$ for some~$i$, then
  $A\leq_{\rh}A,0$ and $A-B\leq_{\rh}A-B,-B$. Thus
  \[
    A-B\leq_{\rh[\!-a_i,\pm x_1,\dots,\pm
      x_m]}\!A-B,-B\leq_{\rh[\!-a_i,\pm x_1,\dots,\pm
      x_m]}\!0\quad\text{for $i=1,\dots,k$.}
  \]
  If $0\rh a_1$, \dots, $0\rh a_k$, then $0\leq_{\rh}A,0$ and
  $-B\leq_{\rh}A-B,-B$. Thus
  \[
    \begin{aligned}
      -B&\leq_{\rh[\!a_1,\dots,a_k,\pm x_1,\dots,\pm x_m]}\!0\\
      A-B&\leq_{\rh[\!a_1,\dots,a_k,\pm x_1,\dots,\pm x_m]}\!A\text.
    \end{aligned}
  \]
  As $A,A-B\rh[\!\pm y_1,\dots,\pm y_n]\!0$, we have
  \[
    A-B\leq_{\rh[\!a_1,\dots,a_k,\pm x_1,\dots,\pm x_m,\pm
      y_1,\dots,\pm y_n]}\!0\text.
  \]
  All together, we conclude that
  \[
    A-B\rh[\!\pm a_1,\dots,\pm a_k,\pm x_1,\dots,\pm x_m,\pm
    y_1,\dots,\pm y_n]\!0\text.
  \]
  
  \emph{Let us prove that the regularisation is regular}, i.e.\ that
  $x+a,y+b\vda[\rh]x+b,y+a$ holds for all $a,b,x,y\in G$: it suffices
  to note that
  \[
    \begin{gathered}[b]
      \text{if $a-b\rh 0$, then $a-b,x-y,y-x,b-a\rh 0$;}\\
      \text{if $b-a\rh 0$, then $a-b,x-y,y-x,b-a\rh 0$.}
    \end{gathered}\qedhere
  \]
\end{proof}

The following lemma justifies the terminology of \cref{ideflorrel}.
One may formulate it as follows: ``regularisation leaves a regular
entailment relation unchanged''.
\begin{lemma}
  \label{lem2ithmDivLorsi}
  Let \(G\) be an ordered group and \(\vda\)~a regular entailment
  relation for~\(G\). Let \(\rh[\vda]\)~be the \si given as the
  restriction of~\(\vda\) to~\(\Pfs(G)\times G\). Then
  \(\vda\)~coincides with the regularisation of~\(\rh[\vda]\).
\end{lemma}

\begin{proof}
  This is a consequence of \cref{+x-x} and \cref{ABA-B} of
  \cref{cor+x-x}.
\end{proof}

\begin{theorem}\label{ithmDivLorsi}
  Let \(\rh\)~be an \si for an ordered group~\(G\). The
  regularisation~\(A\vda[\rh]B\) given in \cref{ideflorrel} is the
  finest regular entailment relation for~\(G\) whose restriction
  to~\(\Pfs(G)\times G\) is coarser than~\(\rh\).
  \end{theorem}

\begin{proof}
  \Cref{lem1ithmDivLorsi} tells that $\vda[\rh]$~is a regular
  entailment relation, and it is clear from its definition that its
  restriction to~$\Pfs(G)\times G$ is coarser than~$\rh$. Now
  let~$\vda$ be a regular entailment relation whose
  restriction~$\rh[\vda]$ to~$\Pfs(G)\times G$ is coarser
  than~$\rh$. Then the same holds for their regularisation, i.e., by
  \cref{lem2ithmDivLorsi}, $\vda$ is coarser than~$\vda[\rh]$.
\end{proof}

These results give rise to the following construction and \lcnamecref{ithGOEntrelGRL}, that one can find in \citealt[\S~2 and page~23]{Lor1953}.
\begin{definition}
  \label{defLorgroup}
   Let $\rh$~be an \si for an ordered group~$G$. The \emph{Lorenzen group} associated with~$\rh$ is the \grl provided by \cref{ithmDivLorsi,ithGOEntrelGRL}.
\end{definition}

%r
%:     Remark{remLorGroup}
\begin{comment} \label{remLorGroup}
\citet[\S~4]{Lor1939} and \citet[II, \S~2, 2]{Jaf1960} follow the Prüfer approach (see \cref{defiGRLor}) for defining the Lorenzen group associated with an \si. The present approach leading to \cref{defLorgroup} is inspired by \citet[\S~2]{Lor1953}. The two definitions are equivalent according to \cref{propGRLor}.\eoe
\end{comment}
%----------- fin remark ----------------------------------

\begin{theorem}
\label{embedLorgroup}
  Let \(\rh\)~be an \si for an ordered group~\(G\). If \(G\)~is \(\rh\)-closed, then \(G\)~embeds into the Lorenzen group associated with~\(\rh\).
\end{theorem}

\subsection{The regularisation of the finest
  \si}\label{subseclcdalaLor}

We shall now give a precise description of the
regularisation~$\vda[{\rh[\mathrm s ]}]$ of the finest \si introduced
in \cref{comment-s-system1}.

\begin{lemma}
  \label{lem:exa}
  Let \((G,{\leq_G})\) be an ordered group and \(\vda\)~a regular
  entailment relation for~\(G\). Let \(a_1,\dots,a_k\in G\). If
  \begin{equation}
    n_1a_1+\dots+n_ka_k\leq_G 0\quad\text{for some integers~\(n_i\geq 0\) not all zero,}\label{eqlemmath2GOEntrelGRL02}
  \end{equation}
  then \(a_1,\dots,a_k\vda0\).
\end{lemma}

\begin{proof}
  This follows from the argument of \cref{exaideflorrel} because of
  \cref{+x-x}.
\end{proof}

Let us write~\(A^{(n)}\) for the \(n\)\/fold sum of the set~\(A\) with
itself: \(A^{(n)}=A+\cdots+A\) (\(n\) times).

\begin{proposition}
  \label{lemmath2GOEntrelGRL0}
  Let \((G,{\leq_G})\)~be an ordered group.  T.f.a.e.\ for
  \(A\in\Pfs(G)\).
  \begin{asparaenum}
  \item\label{lemmath2GOEntrelGRL01} \(A\vda[{\rh[\mathrm s ]}]0\).
  \item\label{lemmath2GOEntrelGRL02} There is an
    integer~\(n\geq1\) such that \(A^{(n)}\rh[\mathrm{s}]0\), i.e.\
    there are \(a_1,\dots,a_k\in A\) such that
    \cref{eqlemmath2GOEntrelGRL02} holds.
  \end{asparaenum}
\end{proposition}

\begin{proof}
  Let us denote Property~\cref{eqlemmath2GOEntrelGRL02} by
  $\varrho(a_1,\dots,a_k)$.
  \begin{asparaitem}
  \item[$(\ref{lemmath2GOEntrelGRL01})\implies(\ref{lemmath2GOEntrelGRL02})$.]
    \Cref{lemrhgamma} shows that if $A\rh[\epsilon_1x_1,\dots,\epsilon_mx_m]0$, then $A,{A+\epsilon_mx_m},\dots,{A+p\epsilon_mx_m}\rh[\epsilon_1x_1,\dots,\epsilon_{m-1}x_{m-1}]0$ for some integer~$p$. One may therefore proceed by induction on~$m$. Firstly, it is clear that $a_1,\dots,a_k\rh[\mathrm s ]0$ implies
    that $\varrho(a_1,\dots,a_k)$ holds. Secondly, suppose that for
    some integers $p$ and~$q$,
    \[
      \begin{aligned}
        &\varrho(a_1,\dots,a_k,a_1+x_m,\dots,a_k+x_m,\dots,a_1+px_m,\dots,a_k+px_m)\text{ and}\\
        &\varrho(a_1,\dots,a_k,a_1-x_m,\dots,a_k-x_m,\dots,a_1-qx_m,\dots,a_k-qx_m)\text{ hold.}
      \end{aligned}
    \]
    Let us show that $\varrho(a_1,\dots,a_k)$ holds. The hypothesis implies that
    there are integers~$n_i,n\geq 0$, at least one~$n_i$ nonzero, such
    that $n_1a_1+\dots+n_ka_k+nx_m\leq_G0$, and integers~$n'_i,n'\geq 0$,
    at least one $n'_i$ nonzero, such that
    $n'_1a_1+\dots+n'_ka_k-n'x_m\leq_G0$. If $n=0$ or if $n'=0$, then we are
    done; otherwise, ${(n'n_1+nn'_1)a_1}+\dots+{(n'n_k+nn'_k)a_k}\leq_G0$
    with at least one~$n'n_i+n n'_i$ nonzero.
  \item[($\ref{lemmath2GOEntrelGRL02})\implies(\ref{lemmath2GOEntrelGRL01})$.]
    This follows from \cref{lem:exa}.\qedhere
  \end{asparaitem}
\end{proof}

\begin{corollary}\label{th2GOEntrelGRL0}
  Let \(G\)~be an ordered group.  The regularisation of the finest \si
  for~\(G\) is the finest regular entailment relation for~\(G\).
\end{corollary}

\begin{proof}
  This follows from \Cref{lem:exa,lemmath2GOEntrelGRL0} because of
  \cref{ABA-B} of \cref{cor+x-x}.
\end{proof}

\begin{corollary}\label{sclosed}
  An ordered group~\((G,{\leq_G})\) is \(\rh[\mathrm s ]\)-closed if and only if
  \begin{equation}
    \text{\(0\leq_Gna\) implies \(0\leq_Ga\)\quad(\(a\in G\), \(n>1\)).}\label{eq:3}
  \end{equation}
\end{corollary}

\begin{corollary}\label{corlem1thGOEntrelGRL}
  Let \((G,{\leq_G})\)~be an ordered group. T.f.a.e.
  \begin{asparaenum}
  \item \(A\vda[{\rh[\mathrm s ]}]B\).
  \item\label{corlem1thGOEntrelGRL2} There is an integer~\(n\geq1\)
    such that for some elements~\(a^{(n)}\in A^{(n)}\)
    and~\(b^{(n)}\in B^{(n)}\), \(a^{(n)}\leq_Gb^{(n)}\) holds.
  \end{asparaenum}
\end{corollary}

\begin{proof}
  Suppose that there are \(a_1,\dots,a_k,b_1,\dots,b_\ell\in G\) such
  that
  $a^{(n)}=n_1a_1+\dots+n_ka_k\leq_Gb^{(n)}=m_1b_1+\dots+m_\ell
  b_\ell$ with integers~\(n_i,m_j\geq0\) such that
  \(n_1+\dots+n_k=m_1+\dots+m_\ell=n\). By the Riesz refining lemma
  \citep[see e.g.][Theorem XI-2.11]{CACM}, there are
  integers~$p_{ij}\geq 0$ such that $n_i=\sum_{j=1}^\ell p_{ij}$ for
  each~$i$ and $m_j=\sum_{i=1}^kp_{ij}$ for each~$j$, so that
  \cref{corlem1thGOEntrelGRL2} may be written
  \[
    \ndsp\sum_{i=1}^k\sum_{j=1}^\ell p_{ij}(a_i-b_j)\leq_G0\text{ for
      some integers $p_{ij}\geq0$ not all zero.}
  \]
  The equivalence follows by \cref{lemmath2GOEntrelGRL0}.
\end{proof}

\subsection{The \grl freely generated by an ordered
  group}\label{sec:lorenz-cliff-dieud}

As an application, we provide the following description for the \grl
freely generated by an ordered group.

\begin{theorem}\label{ithmgogrlfree}
  For every ordered group~\(G\) we can construct an \grl~\(H\) with a
  morphism~\(\mor\colon G\to H\) such that \(0\leq_H\mor(a)\) holds if
  and only if \(0\leq_Gna\) for some~\(n\geq1\). More precisely, \(H\)
  is the \grl freely generated by~\(G\) (in the sense of the left
  adjoint functor of the forgetful functor) and can be constructed as
  the Lorenzen group associated with the finest \si, characterised by:
  \({\mor(a_1)\vii\dots\vii\mor(a_k)}\leq_H\mor(b_1)\vuu\dots\vuu\mor(b_\ell)\)
  holds if and only if there are
  integers~\(n_1,\dots,n_k,m_1,\dots,m_\ell\geq 0\) with
  \(n_1+\cdots+n_k=m_1+\dots+m_\ell\geq1\) such that
  \(n_1a_1+\dots+n_ka_k\leq_Gm_1b_1+\dots+m_\ell b_\ell\).
\end{theorem}

\Cref{ithmgogrlfree} is in fact a reformulation of the following
proposition, enriched with an account of
\cref{th2GOEntrelGRL0,corlem1thGOEntrelGRL,sclosed}.

\begin{proposition}
  Let \((G,{\leq_G})\)~be an ordered group. The Lorenzen group
  associated with the finest \si for~\(G\) is the \grl freely
  generated by~\((G,{\leq_G})\) (in the sense of the left adjoint
  functor of the forgetful functor).
\end{proposition}

\begin{proof}
  The finest monoid of ideals for \(G\) is the meet-monoid~$M$ freely
  generated by~\(G\), and its Grothendieck \grl~$H$ is the \grl freely
  generated by~$M$: therefore $H$~is the \grl freely generated
  by~\(G\) as a monoid, and therefore also as a group.
\end{proof}

\Cref{ithmgogrlfree} may be seen as a generalisation of the following
corollary, the constructive core of the classical
Lo\-ren\-zen-Clif\-ford-Dieu\-don\-né theorem.

%c
%:     Corollary{corthmgogrlfree}
\begin{corollary}[Lorenzen-Clifford-Dieudonné, see {\citealp[Satz~14
    for the $\mathrm{s}$-\scentrel]{Lor1939};
    \citealp[Theorem~1]{Cli1940};
    \citealp[Section~1]{Die1941}}]\label{corthmgogrlfree}
  The ordered group~\((G,{\leq_G})\)~is embeddable into an \grl if and
  only if
  \begin{equation*}
    \text{\(0\leq_Gna\) implies \(0\leq_Ga\)\quad(\(a\in G\), \(n>1\)).}  \tag{\ref{eq:3}}\label{re:eq:3}
  \end{equation*}
\end{corollary}
%--------- fin corollary -------------------------------
%
\begin{proof}
  The condition is clearly necessary. \Cref{ithmgogrlfree} shows that
  it yields the injectivity of the morphism \hbox{$\mor\colon G\to H$}
  as well as the fact that $\mor(x)\leq_H \mor(y)$ implies $x\leq_Gy$.
\end{proof}
\begin{comments}
  \phantomsection\label{comment-s-system}
  \begin{asparaenum}
  \item In each of the three references given in
    \cref{corthmgogrlfree}, the authors invoke a maximality argument
    for showing that $G$~embeds in fact into a direct product of
    linearly ordered groups. The goal of \citet[\S~4;
    {\citeyear{Lor1953}}]{Lor1950} is to avoid the necessarily
    nonconstructive reference to linear orders in conceiving
    embeddings into an \grl, and this endeavour culminates in
    \cref{corthGOEntrelGRL}.
  \item The reader will recognise Condition~\cref{re:eq:3} of
    $\rh[\!\mathrm s]$-closed\-ness of \cref{sclosed} in the condition
    of embeddability stated here.  In fact, in his Ph.D.\ thesis
    \citeyearpar{Lor1939}, \citeauthor{Lor1939} proves
    \cref{corthmgogrlfree} as a side-product of his enterprise of
    generalising the concepts of multiplicative ideal theory to the
    framework of preordered groups.  He is following the Prüfer
    approach presented in \cref{secGoembedsGrl}, in which
    $\rh[\!\mathrm s]$-closedness is being introduced according to
    \cref{def-closed} and the equivalence with
    Condition~\cref{re:eq:3} is easy to check (see
    \citealp[page~358]{Lor1939}, or \citealp[I, \S~4,
    Théorème~2]{Jaf1960}).\eoe
  \end{asparaenum}
\end{comments}

\subsection{The regularisation of the system of Dedekind ideals}\label{subsecaritalaLor}

Let us resume \cref{sec:forc-posit-an-1} with a crucial lemma.

%l
%:     Lemma{lemLordivgroup}
\begin{lemma}\label{lemLordivgroup}
  One has \(A\vda[{\rh[\mathrm d ]}]1\) if and only if
  \(\gen{A}_{\id[A]}\ni 1\).
\end{lemma}
%----------- fin lemma -----------------------------------
%
\begin{proof}
  Suppose that $A\vda[{\rh[\mathrm d ]}]1 $, i.e.\ by
  \cref{forcing-dedekind} that there are elements $x_1,\dots,x_n\in G$
  such that $\gen{A}_{\id[x_1^{\pm1},\dots,x_n^{\pm1}]}\ni 1$. It
  suffices to prove the following fact and to use it in an induction
  argument: suppose that $\gen{A}_{\id[A,x]}\ni 1$ and
  $\gen{A}_{\id[A,x^{-1}]}\ni 1$; then $\gen{A}_{\id[A]}\ni 1$. In fact,
  the hypothesis means that $\gen{A,Ax,\dots,Ax^p}_{\id[A]}\ni 1$ and
  $\langle A,Ax^{-1},\dots,Ax^{-p}\rangle_{\id[A]}\ni 1$ for some~$p\geq0$,
  which implies that
  \[
    \forall k\in\lrb{-p..p}\quad \big\langle
    Ax^{-p},\dots,Ax^{-1},A,Ax,\dots,Ax^p\big\rangle_{\id[A]}\ni
    x^k\text,
  \]
  i.e.\ that there is a matrix~$M$ with coefficients
  in~$\gen{A}_{\id[A]}$ such that $M(x^k)_{-p}^p=(x^k)_{-p}^p$, i.e.\
  $(1-M)(x^k)_{-p}^p=0$. Let us now apply the determinant trick:
  multiplying~$1-M$ by the matrix of its cofactors and expanding it
  yields that $\gen{A}_{\id[A]}\ni 1$.

  Conversely, let $a_1,\dots,a_k$ be the elements of~$A$. For
  each~$i$, $a_ia_i^{-1}=1$, so that $\gen{A}_{\id[a_i^{-1}]}\ni 1$ and
  $A\mathrel{{(\rhd_{\mathrm{d}})}_{a_1^{\pm1},\dots,a_k^{\pm1}}}1$
  for every choice of signs with at least one negative sign: the only
  missing choice of signs consists in the hypothesis
  $\gen{A}_{\id[A]}\ni 1$.
\end{proof}

An element $b\in K$ is said to be \emph{integral over the ideal}
$\gen{A}_\id$ when an integral dependence relation
$b^p=\sum_{k=1}^{p} c_k b^{p-k}$ with $c_k\in{\gen{A}_\id}^{k}$
holds for some~$p\geq1$. If $A=\so1$, then this reduces to the same integral dependence
relation with $c_k\in\id$, i.e.\ to $b$~being integral over~$\id$.

Note that if $A$ contains nonintegral elements, i.e.\ elements not
in~$\id$, then ${\gen{A}_\id}^2$ may or may not be contained in
$\gen{A}_\id$: consider respectively e.g.\ the ideal
$\gen{1,\frac ut}$ in $k[T,U]/(T^3-U^2)=k[t,u]$ and ideals in a Prüfer
domain.

\begin{theorem}[{\citealp[Satz~2]{Lor1953}}]\label{TLordivgroup}
  Let \(\id\) be an integral domain and \(\rh[\mathrm d]\) its system
  of Dedekind ideals.
  \begin{asparaenum}
  \item\label{TLordivgroup1} One has \(A\vda[{\rh[\mathrm d ]}]b\)---i.e.\ the element~\(b\) is
    \(\rh[\mathrm d]\)-dependent on~\(A\); there are
    \(x_1,\dots,x_n\) such that
    \( \gen{A}_{\id[x_1^{\pm1},\dots,x_n^{\pm1}]}\ni b\) for every
    choice of signs---if and only if \(b\)~is integral over the
    ideal~\(\gen{A}_\id\).
  \item\label{TLordivgroup2} One has\/
    $A\vda[{\rh[\mathrm d ]}]B$---that is, there are\/
    $x_1,\dots,x_n$ such that\/
    $\gen{AB^{-1}}_{\id[x_1^{\pm1},\dots,x_n^{\pm1}]}\ni 1$ for every
    choice of signs---if and only if\/
    $\sum_{k=1}^p{\langle AB^{-1}\rangle_\id}^{k}\ni 1$ for
    some\/ $p\geq1$, i.e.\ there is an equality\/ $\som_{k=1}^p f_k=1$
    with each\/ $f_k$ a homogeneous polynomial of degree\/ $k$ in the
    elements of\/ $AB^{-1}$ with coefficients in\/ $\id$.
  \item\label{TLordivgroup3} The divisibility group~\(G\) is
    \(\rh[\!\mathrm d]\)-closed, i.e.\ the
    equivalence
    \[{a\vda[{\rh[\!\mathrm d]}]b} \enskip\iff\enskip a\text{ divides
        \(b\)}\] holds, if and only if \(\id\) is integrally closed.
  \end{asparaenum}
\end{theorem}

\begin{proof}
  \begin{asparaenum}
  \item[(\ref{TLordivgroup1}--\ref{TLordivgroup2})]\refstepcounter{enumi}\refstepcounter{enumi}
    This follows from the previous lemma because
    \[
      \begin{aligned}
        A\vda[{\rh[\mathrm d ]}]b\enskip&\iff\enskip Ab^{-1}\vda[{\rh[\mathrm d ]}]1\text,\\
        \ndsp\sum_{k=1}^{p} c_k b^{p-k}=b^p\text{ with }c_k\in{\gen{A}_\id}^{k}\enskip&\ndsp\iff\enskip\sum_{k=1}^p{\langle Ab^{-1}\rangle_\id}^{k}\ni 1\text,\\
        \ndsp\gen A_{\id[A]}\ni 1\enskip&\ndsp\iff\enskip\exists
        p\geq1\enskip\sum_{k=1}^p{\gen{A}_\id}^{k}\ni 1\text.
      \end{aligned}
    \]
  \item[(\ref{TLordivgroup3})] $\rh[\!\mathrm d]$-closedness is
    equivalent to $1\vda[{\rh[\mathrm d ]}]b\implies \id\ni b$; by
    \cref{TLordivgroup1}, $1\vda[{\rh[\mathrm d ]}]b$ holds if and
    only if $b$~is integral over~$\id$.\qedhere
    \end{asparaenum}
\end{proof}

\subsection{The Lorenzen divisor group of an integral
  domain}\label{secalaLor}

In this section, we note consequences of
\cref{ithmDivLorsi,ithGOEntrelGRL} for Lorenzen's theory of
divisibility presented in \cref{subsecaritalaLor}.

%d
%:     Definition{defiLordivgroup}
\begin{definition}\label{defiLordivgroup}
  Let $\id$~be an integral domain. The \emph{Lorenzen divisor group}
  $\mathrm{Lor}(\id)$ of $\id$ is the Lorenzen group associated by
  \cref{defLorgroup} with the system of Dedekind
  ideals~$\rh[\mathrm{d}]$ for the divisibility group of~$\id$.
\end{definition}
% ----------- fin definition --------------------------------

The following version of \cref{ithmDivLorsi,embedLorgroup} takes into
account the informations provided by \cref{TLordivgroup};
\cref{thmDivLor2-1} emphasises the fact that a regular \entrel is
characterised by its restriction to~$\Pfs(G)\times G$ (\cref{ABA-B} of
\cref{cor+x-x}).

%t
%:     Theorem{thmDivLor2}
\begin{theorem}\label{thmDivLor2}
  Let \(\id\) be an integral domain with field of fractions~\(K\) and
  divisibility group \(G=K\eti/\id\eti\). The
  \entrel~\(\vda[{\rh[\!\mathrm d]}]\) generates the Lorenzen divisor
  group \(\mathrm{Lor}(\id)\) together with a morphism of ordered
  groups \(\mor\colon G\to\mathrm{Lor}(\id)\) that satisfies the
  following properties.
  \begin{asparaenum}
  \item\label{thmDivLor2-1} The \gui{ideal Lorenzen gcd} of
    \(a_1,\dots,a_k\in K\etl\) is characterised by
    \[
      \mor(a_1)\vii\dots\vii \mor(a_k) \leq \mor (b)\enskip \iff
      \begin{aligned}[t]
        &b\text{ is integral over}\\
        &\text{the ideal \(\gen{a_1,\dots,a_k}_\id\).}
      \end{aligned}
    \]
  \item\label{thmDivLor2-2} The morphism \(\mor\) is an embedding if
    and only if \(\id\) is integrally closed.
  \end{asparaenum}
\end{theorem}

\Cref{thmDivLor2-1} lends itself to an extensional formulation in
terms of the integral closure $\Icl_K(a_1,\dots,a_k)$ of ideals
\(\gen{a_1\dots,a_k}_\id\) in the field of fractions~$K$. If
\(a_1,\dots,a_k\in \id\etl\), i.e.\ if one considers integral finitely
generated ideals, it seems more appropriate to find a formulation in
terms of the integral closure $\Icl(a_1,\dots,a_k)$ in the integral
domain. This works because the elements $a_1,\dots,a_k,b\in K\etl$ in a
relation $a_1,\dots,a_k\vda[{\rh[\!\mathrm d]}]b$ may be translated by
an~$x$ into~$\id\etl$. This yields the following theorem, in which we
use the conventional additive notation for divisor groups of an
integral domain. It takes into account the construction of the
Lorenzen group as the Grothendieck \grl of the meet-monoid associated
with the regularisation of the system of Dedekind ideals in the proof
of \cref{ithGOEntrelGRL}, i.e.\ as formal differences
$\Vii\mor(A)-\Vii\mor(B)$; we take advantage of the fact that
$\Vii\mor(A)-\Vii\mor(B)=\Vii\mor(xA)-\Vii\mor(xB)$ for every~$x$, so
that it suffices to use integral ideals in this construction.

%c
%:     Corollary{corthmDivLor2}
\begin{theorem}\label{corthmDivLor2}
  Let \(\id\) be an integral domain. The Lorenzen divisor group
  \(\mathrm{Lor}(\id)\) can be realised extensionally in the following
  way.
  \begin{asparaitem}
  \item A \emph{basic divisor} is realised as the integral closure
    \(\Icl(a_1\dots,a_k)\) of an ordinary, i.e.\ integral finitely
    generated ideal \(\gen{a_1\dots,a_k}_\id\) with
    \(a_1,\dots,a_k\in \id\etl\).
  \item The neutral element of the group, i.e.\ the divisor~\(0\), is
    realised as \(\Icl({1})\).
  \item The meet of two basic divisors is realised as
    \[
      \Icl(a_1,\dots,a_k)\vii\Icl(b_1,\dots,b_\ell)=\Icl(a_1,\dots,a_k,b_1,\dots,b_\ell).
    \]
  \item The sum of two basic divisors is realised as
    \[
      \Icl(a_1,\dots,a_k)+\Icl(b_1,\dots,b_\ell)=\Icl(a_1b_1,\dots\dots,a_kb_\ell).
    \]
  \item The order relation between  basic divisors is realised as
    \[
      \Icl(a_1,\dots,a_k)\leq \Icl(b_1,\dots,b_\ell) \enskip
      \iff\enskip \Icl(a_1,\dots,a_k)\supseteq \Icl(b_1,\dots,b_\ell).
    \]
    In particular, \(\Icl(a)\leq\Icl(b)\) holds if and only if \(b\) is integral over \(\gen{a}_\id\).
  \item Every divisor is realised as the formal difference of two
    basic divisors.
\end{asparaitem}

\end{theorem}
% --------- fin corollary -------------------------------

\begin{remarks}\phantomsection\label{remthmDivLor2}
  \begin{asparaenum}
  \item This theorem holds without condition of integral closedness,
    but beware of the following fact: if some
    $b\in K\etl\setminus\id\etl$ is integral over~$\id$, then
    $\Icl_K(1)\ni b$ and $0\leq\mor(b)$; however $\Icl(1)=\id$ and
    $\mor(b)$ is realised as a nonbasic divisor. An example for this
    is $\id=\QQ[t^2,t^3]$, $b=\frac{t^3}{t^2}$,
    $\mor(b) = \mor(t^3) - \mor(t^2)$, $\Icl(t^3)=\gen{t^3,t^4}_\id$,
    $\Icl(t^2)=\gen{t^2,t^3}_\id$.
  \item If every positive divisor is basic, then one can show the
    domain to be Prüfer.
  \item When $\id$ is a Prüfer domain, the Lorenzen divisor group
    $\mathrm{Lor}(\id)$ coincides with the usual divisor group, the
    group of finitely generated fractional ideals defined by Dedekind
    and Kronecker.  In fact, all finitely generated ideals are
    integrally closed in a Prüfer domain, so that
    $\Icl(a_1,\dots,a_k)=\gen{a_1,\dots,a_k}_\id$.
  \item The integral domain $\id=\QQ[t,u]$ is a gcd domain of
    dimension $\geq 2$, so that its divisibility group $G$ is an
    \grl. The domain $\id$ is not Prüfer and the Lorenzen divisor
    group is much greater than $G$: e.g.\ the ideal gcd of $t^3$ and
    $u^3$ in $\mathrm{Lor}(\id)$ corresponds to the integrally closed
    ideal $\gen{t^3,t^2u,tu^2,u^3}$, whereas their gcd in $\id\etl$ is
    $1$, corresponding to the ideal~$\gen{1}$.  In this case, we see
    that $G$ is a proper quotient of $\mathrm{Lor}(\id)$.\eoe
  \end{asparaenum}
\end{remarks}

The following corollary concentrates upon the cancellation property
holding in \grls.  Note that the integral closure of an integral
finitely generated ideal in an integrally closed integral domain is
equal to its integral closure in the field of fractions.

% :     corollary{corthmDivLor}
\begin{corollary}[see {\citealp[pages~108--109]{Macaulay}}]\label{corthmDivLor}
  Let \(\id\)~be an integrally closed integral domain. When \(\fa\)~is
  a finitely generated integral ideal~\(\gen{a_1,\dots,a_k}_\id\) with
  \(a_1,\dots,a_k\in \id\etl\), we let
  \(\overline\fa=\Icl(a_1,\dots,a_k)\)~be the integral closure
  of~\(\fa\). Then, if \(\fa\),~\(\fb\), and~\(\fc\) are nonzero
  finitely generated integral ideals, we have the cancellation \property
  \[
    \overline{\fa\,\fb}\supseteq\overline{\fa\,\fc}\enskip\implies\enskip
    \overline\fb\supseteq \overline\fc.
  \]
\end{corollary}

This corollary is a key result for ``containment in the wider sense''
as considered by Leopold \citet{kronecker83} (see \citet{penchevre},
pages 36--37). H. S.~\citet{Macaulay} gives a proof based on the
multivariate resultant. We may also deduce it as a consequence of
Prüfer's \cref{thSIPrufer} (see \cref{lemDivLor2item2} of
\cref{lemDivLor2}, compare \citealp[\S~6]{Pru1932},
\citealp[Nr.~46]{Kru35}).

%%%%%%%%%%%%%%%%%%%%%%%%%%%%%%%%%%%%%%%%%%%%%%%%%%%%%%%%%%%%%%%%%%%%
%%%%%%%%%%%%%%%%%%%%%%%%%%%%%%%%%%%%%%%%%%%%%%%%%%%%%%%%%%%%%%%%%%%%
\section{\Sis and Prüfer's theorem}\label{secGoembedsGrl}

In this section, we account for another way to obtain the Lorenzen
group associated with an \si for an ordered group
(\cref{defLorgroup}). This way has historical precedence, as it dates
back to \citeauthor{Lor1939}'s Ph.D.\ thesis \citeyearpar{Lor1939},
that builds on earlier work by \citet{Pru1932}.  In the case of the
system of Dedekind ideals, this approach provides another way of
understanding the Lorenzen divisor group of an integral domain.

\subsection{Prüfer's properties \texorpdfstring{$\Gamma$}{Gamma} and
  \texorpdfstring{$\Delta$}{Delta}}

Let us now express cancellativity of the \mmonoid as a property of the
\si itself (a.k.a.\ ``endlich arithmetisch brauchbar'', ``e.a.b.'',
see \cref{remhist3}), as in \citealt[\S~3]{Pru1932}.
%l
%:     Lemma{lemSIJaf}
\begin{lemma}[Prüfer's \Property~$\Gamma$ of cancellativity]\label{lemSIJaf}
  Let \(\rh\)~be an \si for an ordered group~\(G\). The associated
  \mmonoid~\(M\) is cancellative, i.e.\ \(\Vii(A+X)=_M\Vii(B+X)\)
  implies \(\Vii A=_M\Vii B\), if and only if the following property
  holds:
  \begin{equation}
     A+ X\leq_{\rh}b+ X \enskip\implies\enskip A\rh b\text.\label{propGamma}
  \end{equation}
  This holds if and only if
  $ A+ X\leq_{\rh} X\implies A\rh 0$.
\end{lemma}
%----------- fin lemma -----------------------------------
%
\begin{proof}
  The second implication, a particular case of the first one, implies
  the first one by equivariance. Let us work with the first
  implication.  Cancellativity means that if
  $ A+ X\leq_{\rh} B+ X$, then
  $ A\leq_{\rh} B$.  \Property~\cref{propGamma} is necessary:
  take~\hbox{$B=\so b$}.  Let us show that it is sufficient. Assume
  \hbox{$ A+ X\leq_{\rh} B+ X$} and let $b\in B$. As
  $B\rh b$, we have ${ B+ X}\leq_{\rh}{b+ X}$, whence
  ${ A+ X}\leq_{\rh}{b+ X}$. So $A\rh b$. Since this holds
  for each $b\in B$, we get $ A\leq_{\rh} B$.\qedhere
\end{proof}
\begin{remark}\label{remGamma}
  The original version of Prüfer's Property~$\Gamma$ states, for a
  set-theoretical star-ope\-ra\-tion $A\mapsto A_r$ on nonempty
  finitely enumerated subsets of~$G$ as considered in \cref{remJafsi2}
  of \cref{remJafsi}, the cancellation property
  $(A+X)_r\supseteq(B+X)_r\implies A_r\supseteq B_r$.\eoe
\end{remark}

Prüfer's \cref{thSIPrufer} will reveal the significance of the
following definition. We shall check in \cref{propGRLor} that it
agrees with \cref{ideflorrel}.

\begin{definition}[Prüfer's \Property~$\Delta$ of integral
  closedness]\label{def-closed}
  Let $\rh$ be an \si for an ordered group~$G$.  The group~$G$ is
  \emph{$\rh$-closed} if
  ${ X\leq_{\rh}b+ X}\implies{0\leq_G b}$.
\end{definition}

\begin{remark}
  The original version of Prüfer's Property~$\Delta$ states the
  cancellation property ${X_r\supseteq b+X_r}\implies{0\leq_G b}$.\eoe
\end{remark}

%%%%%%%%%%%%%%%%%%%%%%%%%%%%%%%%%%%%%%%%%%%%%%%%%%%%%%%%%%%%%%%%%%%%
\subsection{Forcing cancellativity: Prüfer's theorem}

When the monoid~$M$ in \cref{ThSIJaf} is not cancellative, it is
possible to adjust the \si in order to straighten the situation.  A
priori, it suffices to consider the Grothendieck \grl of~$M$
(\cref{thGoembedsGrl}).  But we have to see that this corresponds to
an \si for~$G$, and to provide a description for it.
The following theorem is a reformulation of Prüfer's theorem
\citep[\S~6]{Pru1932}. We follow the proofs in
\citealt[pages~42--43]{Jaf1960}.  In fact, the language of \scentrels
simplifies the proofs. Jaffard's statement corresponds to
\cref{thSIPrufer1,thSIPrufer4}, and
\cref{thSIPrufer2,thSIPrufer3} have been added by us.

%t
%:     Theorem{thSIPrufer}
\begin{theorem}[Prüfer's theorem]\label{thSIPrufer}
  Let \(\rh\) be an \si for an ordered group~\(G\). We define the
  relation \(\Vda\) between~\(\Pfs(G)\) and~\(G\) by
  \[
    A\Vda b \enskip\equidef\enskip\exists X\in \Pfs(G)\enskip  A+ X \leq_{\rh}b+ X\text.
  \]
  \begin{asparaenum}
  \item\label{thSIPrufer1} The relation \(\Vda\) is an \si for \(G\),
    and the associated \mmonoid\allowbreak \(M_\mathrm{a}\)
    (\cref{ThSIJaf}) is cancellative.
  \item\label{thSIPrufer2} The \mmonoid \(M_\mathrm{a}\) embeds into
    its Grothendieck \grl~\(H_\mathrm{a}\).
  \item\label{thSIPrufer3} The system~\(\Vda\) is the finest
    \si~\(\rh'\) coarser than~\(\rh\) such that \(M_\mathrm{a}\) is
    cancellative, i.e.\ forcing
    \[
       A+ X\leq_{\rh'}b+ X \enskip\implies\enskip A\rh'
      b\text.
    \]
  \item\label{thSIPrufer4} The implication
    \(a\Vda b\implies a\leq_G b\) holds if (and only if) \(G\)~is
    \(\rh\)-closed (\cref{def-closed}); in this case, \(G\) embeds
    into~\(H_\mathrm{a}\).
  \end{asparaenum}
\end{theorem}
%----------- fin theorem -----------------------------

\begin{proof}
  Note that if $ A+ X\leq_{\rh}b+ X$, then
  $ A+ X+ Y\leq_{\rh}b+ X+ Y$ for all~$Y$ (see the
  proof of \cref{ThSIJaf} on \cpageref{proofThSIJaf}). This makes the
  definition of~$\Vda$ very easy to use.  In the proof below, we have
  two preorder relations on~$\Pfs(G)$ ($\leq_{\rh}$ \hbox{and
    $\leq_\mathrm{a}$}), and we shall proceed as if they were order
  relations (i.e.\ we shall descend to the quotients).

\begin{asparaenum}
\item[(\ref{thSIPrufer1})]
  \leavevmode\enspace\Bullet\enspace\emph{Reflexivity and \presord}
  (of the relation $\Vda$). Setting $X=\so{0}$ in the definition
  of~$\Vda$ shows that $a\leq_G b$ implies $a\Vda b$.
\item[\Bullet\enspace\emph{Monotonicity}.] It suffices to note that
  the elements~${(A,A')+X}$ and~${A+X},{A'+X}$ of $\Pfs(G)$ are the
  same: therefore, if $ A+ X\leq_{\rh}b+ X$, then
  $  (A,A')+ X \leq_{\rh}b+ X$.
\item[\Bullet\enspace\emph{Transitivity}.] Assume $A \Vda c$ and
  $A,c\Vda b$: we have an $X$ such that
  $ A+ X\leq_{\rh}c+ X$ and a $Y$ such that
  $ (A,c)+ Y\leq_{\rh}b+ Y$; these inequalities imply
  respectively $ A+X+Y\leq_{\rh} c+X+Y$ and
  $ A+X+Y, c+X+Y\leq_{\rh} b+X+Y$; we deduce
  $ A+X+Y\leq_{\rh} b+X+Y$, so that $A\Vda b$.
\item[\Bullet\enspace\emph{\Equivariance}.] If $A\Vda b$, we have an
  $X$ such that $ A+ X\leq_{\rh}b+ X$, so that, since
  $\leq_{\rh}$~is \equivariant, $x+ A+ X\leq_{\rh}x+b+ X$.
  This yields $x+ A\Vda x+b$.
\item[\Bullet\enspace\emph{Cancellativity} (of the
  \mmonoid\/~$M_\mathrm{a}$).]  Let us denote by~$\leq_\mathrm{a}$ the
  order relation associated to~$\Vda$. By \cref{lemSIJaf}, it suffices
  to suppose that $ A+ X\leq _\mathrm{a} X$ and to deduce
  that $A\Vda 0$.  But the hypothesis means that $A+X\Vda x$ for each
  $x\in X$, i.e.\ that for each $x\in X$ there is a $Y_x$ such that
  $ A+ X+ Y_x\leq_{\rh}x+ Y_x$.  Let
  $Y=\sum_{x\in X}Y_x$: we have
  $ A+ X+ Y\leq_{\rh}x+ Y$.  As $x\in X$~is arbitrary,
  $ A+ X+ Y\leq_{\rh} X+ Y$: this yields $A\Vda 0$
  as desired.
\item[(\ref{thSIPrufer2})] Follows from \cref{thSIPrufer1} by
  \cref{thGoembedsGrl}.
\item[(\ref{thSIPrufer3})] This is immediate from the definition
  of~$\Vda$: it has been defined in a minimal way as coarser
  than~$\rh$ and forcing the cancellativity of the monoid
  $M_\mathrm{a}$ as characterised in \cref{lemSIJaf}.
\item[(\ref{thSIPrufer4})] If $a\Vda b$, then we have an $X$ such that
  $a+ X\leq_{\rh}b+ X$, so that by a translation
  $ X\leq_{\rh}(b-a)+ X$. The hypothesis on~$G$ yields
  $0\leq_G b-a$. By a translation, we get $a\leq_G b$.\qedhere
\end{asparaenum}
\end{proof}

\begin{comment}
  This is the approach proposed in \citealt[\S~4]{Lor1939}. Lorenzen
  abandoned it in favour of \cref{ideflorrel} for the purpose of
  generalising his theory to noncommutative groups. See also
  \cref{remhist3,comment-s-system}.\eoe
\end{comment}

% d
%:     Definition{defiGRLor}
\begin{definition}[see {\citealp[page~546]{Lor1939}}, or {\citealp[II,
    \S~2, 2]{Jaf1960}}]\label{defiGRLor}
  Let \(\rh\) be an \si for an ordered group~\(G\). The \grl in \cref{thSIPrufer2} of \cref{thSIPrufer} is the
  \emph{Lorenzen group} associated with~$\rh$.
\end{definition}
%----------- fin definition --------------------------------

%:     Proposition{propGRLor}
\begin{proposition}[{\citealp[Satz~27]{Lor1950}}]\label{propGRLor}
  The definition of \(A\Vda0\) in \cref{thSIPrufer} agrees with
  \cref{ideflorrel} of \(A\vda[\rh]0\). So \cref{def-closed} of
  \(\rh\)-closedness agrees with that of \cref{ideflorrel}, and
  \cref{defiGRLor} of the Lorenzen group agrees with that of
  \cref{defLorgroup}.
\end{proposition}
%----------- fin proposition -----------------------------

\begin{proof}
  This proposition expresses that, given an \si~$\rh$ for an ordered
  group $G$ and \hbox{an $A\in\Pfs(G)$}, we have $A\vda[\rh]0$
  (\cref{ideflorrel}) if and only if ${A+X} \leq_{\rh}X$ for
  some~$X\in \Pfs(G)$.  First, $A+Y \leq_{\rh_{\!x}} Y$ and
  $A +Z\leq_{\rh_{\!-x}} Z$ imply $A+X \leq_{\rh}X$ for some~$X$. In
  fact, we have~$p$ and~$q$ such that
  \[
    \begin{aligned}
      A+Y, A+Y + x, \dots, A+Y + px &\leq_{\rh} Y\text{ and}\\
      A+Z, A+Z - x, \dots,  A+Z - qx&\leq_{\rh} Z\text{ hold,}
    \end{aligned}
  \]
  which yield that for $z\in Z$, $j\leq q$, $y\in Y$, and $k\leq p$,
  \[
    \begin{aligned}
      A+Y+z-jx,\dots,A+Y+z+(p-j)x&\leq_{\rh} Y+z-jx\text{ and}\\
      A+y+Z+kx,\dots,A+y+Z+(k-q)x&\leq_{\rh} y+Z+kx\text{ hold,}
    \end{aligned}
  \]
  so that $A + X \leq_{\rh} X$ for $X = Y+Z+\{ -qx,\dots, px \}$.

  In the other direction, assume that $A + X \rh x_i$ for each $x_i$
  in $X=\so{x_1,\dots,x_m}$.  Let $x_{i,j}={x_i-x_j}$ ($i<j\in\lrbm$)
  and let us prove that $A \rh_{\pm x_{1,2},\pm x_{1,3},\dots,\pm x_{m-1,m}} 0$.
  In fact, for any system of constraints
  $(\epsilon_{1,2} x_{1,2},\epsilon_{1,3} x_{1,3},\dots,\epsilon_{m-1,m} x_{m-1,m})$ with
  $\epsilon_{i,j}=\pm1$, the elements $x_i$ are linearly ordered in the associated
  \mmonoid~$M_\epsilon$. E.g.\
  $x_1\leq_{M_\epsilon} x_2\leq_{M_\epsilon} \dots\leq_{M_\epsilon}
  x_m$ holds, in which case
  \[
    \ndsp\Vii(A+x_1,\dots,A+x_m)=_{M_\epsilon}\Vii(A+x_1)\leq_{M_\epsilon}x_1
  \]
  holds, which yields $ \Vii A\leq_{M_\epsilon} 0$ by a
  translation.\qedhere
\end{proof}

\begin{remarks}\phantomsection\label{lemDivLor2}
  \begin{asparaenum}
  \item Informally, the content of this proposition may be expressed as
    follows. By starting from $\rh$ and by adding new pairs~$(A,b)$
    such that $A\rh' b$, on the one side Prüfer forces the
    cancellativity of the \mmonoid $M_\mathrm{a}$, and on the other
    side Lorenzen forces $\rh$ to become the restriction of an \entrel
    (which is still an \si, as follows trivially from Lorenzen's
    definition). In fact, each approach realises both aims, but each
    one realises its own aim in a minimal way. So they give the same
    result.
  \item\label{lemDivLor2item2} \Cref{thSIPrufer} allows one to recover
    the results of \cref{TLordivgroup} and of \cref{thmDivLor2} in the
    Prüfer approach. In particular, one may check that
    $A\mathop{(\rhd_{\!\mathrm d})_{\mathrm a}}b$ holds if and only if
    $b$ is integral over the fractional ideal $\gen{A}_\id$ \citep[by
    applying the determinant trick, see][\S~6]{Pru1932}. One may also
    check that the hypothesis in \cref{thSIPrufer4} of
    \cref{thSIPrufer} holds if and only if $\id$ is integrally
    closed. In this case, the elements $\geq 1$ of the
    \mmonoid~$M_{\mathrm a}$ in \cref{thSIPrufer2} of
    \cref{thSIPrufer} can be identified with the integrally closed ideals
    generated by nonempty finitely enumerated subsets~$A$
    of~$\id\etl$; therefore \cref{thSIPrufer1} of \cref{thSIPrufer}
    yields the cancellation \property stated in
    \cref{corthmDivLor}.\eoe
  \end{asparaenum}
\end{remarks}

\begin{acknowledgement}
  This research has been supported through the program ``Research in
  pairs'' by the Mathematisches Forschungsinstitut Oberwolfach in 2016
  and through the French ``Investissements d'avenir'' program, project
  ISITE-BFC, contract ANR-15-IDEX-03. The second and third authors
  benefitted from the hospitality of the university of Gothenburg for
  leading this research. We also warmly thank the referee for his very
  careful reading.
\end{acknowledgement}

\end{document}